\documentclass[12pt,a4paper,final,pdftex]{article}

\usepackage{mathptmx}       
\usepackage{helvet}         
\usepackage{courier}        

\usepackage[pdftex]{graphicx}        

\usepackage[a4paper, hmargin={3.52cm,3.52cm},vmargin={3.3cm,3.3cm}]{geometry}
\usepackage{amsmath, amssymb, amsfonts, amsbsy, latexsym, color}
\usepackage{amscd,amsxtra}
\usepackage{amsthm}

\usepackage[T1]{fontenc}
\usepackage[latin1]{inputenc}
\usepackage[english]{babel}
\usepackage[labelfont={bf}]{caption,subfig}
\usepackage{cite}
\usepackage{enumerate}
\usepackage{url}
\usepackage{fancyhdr,lastpage}

\usepackage[bottom]{footmisc}
\usepackage{cite}
\usepackage{url}
\usepackage{hyperref}

\def\maketimestamp{\count255=\time
\divide\count255 by 60\relax
\edef\thetime{\the\count255:}%
\multiply\count255 by-60\relax
\advance\count255 by\time
\edef\thetime{\thetime\ifnum\count255<10 0\fi\the\count255}
\edef\thedate{\number\day-\ifcase\month\or Jan\or Feb\or Mar\or
             Apr\or May\or Jun\or Jul\or Aug\or Sep\or Oct\or
             Nov\or Dec\fi-\number\year}
\def\timstamp{\hbox to\hsize{\tt\hfil\thedate\hfil\thetime\hfil}}}
\maketimestamp
\fancypagestyle{plain}{%
\fancyhf{} 
\fancyfoot[C]{\small \thepage{} of \pageref{LastPage}}}

\numberwithin{equation}{section}  
\allowdisplaybreaks[4]
\newtheorem{theorem}{Theorem}[section]

\newtheorem{lemma}[theorem]{Lemma}
\newtheorem{proposition}[theorem]{Proposition}

\theoremstyle{definition}
\newtheorem{definition}{Definition}[section]

\newtheorem{example}{Example}
\newtheorem{setup}{Setup} 
\theoremstyle{remark}
\newtheorem{remark}{Remark}

\newcommand{\bitem}{\begin{itemize}}
\newcommand{\eitem}{\end{itemize}}
\newcommand{\benum}{\begin{enumerate}}
\newcommand{\eenum}{\end{enumerate}}
\newcommand{\beq}{\begin{equation}}
\newcommand{\eeq}{\end{equation}}
\newcommand{\ip}[2]{\langle#1,#2\rangle}

 \def\NN{\mathbb{N}}
 
 \def\RR{\mathbb{R}}

\newcommand\cC{{\mathcal{C}}}
\newcommand\cD{{\mathcal{D}}}
\newcommand\cE{{\mathcal{E}}}

\newcommand\cP{{\mathcal{P}}}

\newcommand\cH{\mathcal{H}}

\newcommand\cL{{\mathcal{L}}}

\newcommand\cQ{{\mathcal{Q}}}
\newcommand\cS{{\mathcal{S}}}

\newcommand\SH{SH}

\DeclareMathOperator{\Span}{span} %
\DeclareMathOperator{\supp}{supp} %
\DeclareMathOperator{\diagonal}{diag} %
\DeclareMathOperator{\area}{area} %
\newcommand{\diag}[1]{\diagonal{(#1)}}

\newcommand{\expo}[1]{\mathrm{e}^{#1}} 

\newcommand{\are}[1]{\area{(#1)}}

\newcommand{\card}[1]{\# \abs{#1}}

\newcommand{\cardsmall}[1]{\# \abssmall{#1}}

\newcommand{\eps}{\ensuremath{\varepsilon}}

\newcommand*{\numbersys}[1]{\ensuremath{\mathbb{#1}}}
\newcommand*{\C}{\numbersys{C}}
\newcommand*{\R}{\numbersys{R}}

\newcommand*{\Z}{\numbersys{Z}}

\newcommand*{\N}{\numbersys{N}}

\newcommand{\itvcc}[2]{\ensuremath{\left[{#1},{#2}\right]}} %
\newcommand{\itvccs}[2]{\ensuremath{\lbrack{#1},{#2}\rbrack}} %
\newcommand{\abs}[1]{\ensuremath{\left\lvert#1\right\rvert}}
\newcommand{\abssmall}[1]{\ensuremath{\lvert#1\rvert}}
\newcommand{\absbig}[1]{\ensuremath{\bigl\lvert#1\bigr\rvert}}
\newcommand{\absBig}[1]{\ensuremath{\Bigl\lvert#1\Bigr\rvert}}
\newcommand{\norm}[2][]{\ensuremath{\left\lVert#2\right\rVert_{#1}}}

\newcommand{\normBig}[2][]{\ensuremath{\Bigl\lVert#2\Bigr\rVert_{#1}}}
\newcommand{\normsmall}[2][]{\ensuremath{\lVert#2\rVert_{#1}}}
\newcommand{\innerprod}[3][]{\ensuremath{\left\langle #2,#3\right\rangle_{\! #1}}}

\newcommand{\innerprods}[2]{\ensuremath{\langle #1,#2\rangle}}

\newcommand{\setprop}[2]{\ensuremath{\left\lbrace{#1} : {#2}\right\rbrace}}

\newcommand{\setpropbig}[2]{\ensuremath{\bigl\lbrace{#1} :
      {#2}\bigr\rbrace}}

\newcommand{\ceil}[1]{\left\lceil #1 \right\rceil}
\newcommand{\ceilsmall}[1]{\lceil #1 \rceil}

\newcommand{\pa}{\partial}

\usepackage{xspace}
\newcommand{\ie}{i.e.,\xspace} 
\newcommand{\eg}{e.g.,\xspace} 

\newcommand{\scp}{\ensuremath\alpha}

\newcommand{\cp}{\ensuremath\nu}
\newcommand{\D}{\ensuremath\mathrm{d}}
\newcommand{\E}{\ensuremath\mathrm{e}}

\makeatletter
\def\blfootnote{\xdef\@thefnmark{}\@footnotetext} 
\def\subjclass{\xdef\@thefnmark{}\@footnotetext}
\long\def\symbolfootnote[#1]#2{\begingroup%
\def\thefootnote{\fnsymbol{footnote}}\footnote[#1]{#2}\endgroup} 
\if@titlepage
  \renewenvironment{abstract}{%
      \titlepage
      \null\vfil
      \@beginparpenalty\@lowpenalty
      \begin{center}%
        \bfseries \abstractname
        \@endparpenalty\@M
      \end{center}}%
     {\par\vfil\null\endtitlepage}
\else
  \renewenvironment{abstract}{%
      \if@twocolumn
        \section*{\abstractname}%
      \else
        \small
        \list{}{%
          \settowidth{\labelwidth}{\textbf{\abstractname:}}
          \setlength{\leftmargin}{50pt}
          \setlength{\rightmargin}{50pt}
          \setlength{\itemindent}{\labelwidth}
          \addtolength{\itemindent}{\labelsep}
        }
        \item[\textbf{\abstractname:}]

      \fi}
      {\if@twocolumn\else\endlist\fi}
\fi
\makeatother

\begin{document}

\title{Shearlets and Optimally Sparse Approximations\footnote{book chapter in \emph{Shearlets: Multiscale Analysis for Multivariate Data}, Birkh\"auser-Springer.}}
\author{Gitta Kutyniok\footnote{Institute of Mathematics, University
    of Osnabr\"uck, 49069 Osnabr\"uck, Germany, E-mail:
\url{kutyniok@math.uni-osnabrueck.de}}, Jakob Lemvig\footnote{Institute of Mathematics, University of Osnabr\"uck,
    49069 Osnabr\"uck, Germany, E-mail:
    \url{jlemvig@math.uni-osnabrueck.de}}, and Wang-Q
    Lim\footnote{Institute of Mathematics, University of Osnabr\"uck,
49069 Osnabr\"uck,~Germany, E-mail: \url{wlim@math.uni-osnabrueck.de}}}

\date{\today}

\maketitle 

\thispagestyle{plain}
\begin{abstract} 
Multivariate functions are typically governed by an\-iso\-tro\-pic features such as edges
in images or shock fronts in solutions of transport-dominated equations. One major goal both
for the purpose of compression as well as for an efficient analysis is the provision of optimally
sparse approximations of such functions. Recently, cartoon-like images were introduced in 2D and
3D as a suitable model class, and approximation properties were measured by considering the
decay rate of the $L^2$ error of the best $N$-term approximation. Shearlet systems are to date
the only representation system, which provide optimally sparse approximations of this model class
in 2D as well as 3D. Even more, in contrast to all other directional representation systems, a
theory for compactly supported shearlet frames was derived which moreover also satisfy this
optimality benchmark. This chapter shall serve as an introduction to and a survey about sparse
approximations of cartoon-like images by band-limited and also compactly supported shearlet frames
as well as a reference for the state-of-the-art of this research field.
\end{abstract}

\section{Introduction}
\label{sec:introduction}

Scientists face a rapidly growing deluge of
data, which requires highly sophisticated methodologies for analysis
and compression. Simultaneously, the complexity of the data is
increasing, evidenced in particular by the observation that data
becomes increasingly high-dimensional. One of the most prominent
features of data are singularities which is justified, for instance,
by the observation from computer visionists that the human eye is
most sensitive to smooth geometric areas divided by sharp edges.
Intriguingly, already the step from univariate to multivariate data
causes a significant change in the behavior of singularities.
Whereas one-dimensional
  (1D) functions can only exhibit point singularities, singularities
of two-dimensional (2D) functions
can already be of both point as well as curvilinear type. Thus, in contrast to {\em isotropic} features -- point singularities --,
suddenly {\em anisotropic} features -- curvilinear singularities -- are possible. And, in fact, multivariate functions are
typically governed by {\em anisotropic phenomena}. Think, for instance, of edges in digital images or evolving shock fronts in
solutions of transport-dominated equations. These two exemplary situations also show that such phenomena occur even for both
explicitly as well as implicitly given data.

One major goal both for the purpose of compression as well as for an efficient analysis is the introduction of representation
systems for `good' approximation of anisotropic phenomena, more precisely, of multivariate functions governed by anisotropic
features. This raises the following fundamental questions:
\begin{enumerate}
\item[(P1)] What is a suitable model for functions governed by anisotropic features?
\item[(P2)] How do we measure `good' approximation and what is a benchmark for optimality?
\item[(P3)] Is the step from 1D to 2D already the crucial step or how does this framework scale with increasing dimension?
\item[(P4)] Which representation system behaves optimally?
\end{enumerate}

Let us now first debate these questions on a higher and more intuitive level, and later on delve into the precise mathematical formalism.

\subsection{Choice of Model for Anisotropic Features}

Each model design has to face the trade-off between closeness to the true situation versus sufficient simplicity to enable
analysis of the model. The suggestion of a suitable model for functions governed by anisotropic features in \cite{Don99}
solved this problem in the following way. As a model for an image, it first of all requires the $L^2(\R^2)$ functions
serving as a model to be supported on the unit square $[0,1]^2$. These functions shall then consist of the minimal
number of smooth parts, namely two. To avoid artificial problems with a discontinuity ending at the boundary of $[0,1]^2$,
the boundary curve of one of the smooth parts is entirely contained in $(0,1)^2$. It now remains to decide upon the
regularity of the smooth parts of the model functions and of the boundary curve, which were chosen to both be $C^2$.
Thus, concluding, a possible suitable model for functions governed by anisotropic features are 2D functions which are
supported on $[0,1]^2$ and $C^2$ apart from a closed $C^2$ discontinuity curve; these are typically referred to as
{\em cartoon-like images} (cf. chapter \cite{Intro}). This provides an answer to (P1). Extensions of this 2D model to
piecewise smooth curves were then suggested in \cite{CD04}, and extensions to 3D as well as to different types of
regularity were introduced in \cite{GL10_3d,KLL10}.

\subsection{Measure for Sparse Approximation and Optimality}

The quality of the performance of a representation system with respect
to cartoon-like images is typically measured by taking a non-linear
approximation viewpoint. More precisely, given a cartoon-like image
and a representation system which forms an orthonormal basis, the
chosen measure is the asymptotic behavior of the $L^2$ error of the
best $N$-term (non-linear) approximation in the number of terms $N$.
This intuitively measures how fast the $\ell^2$ norm of the tail of
the expansion decays as more and more terms are used for the
approximation. A slight subtlety has to be observed if the
representation system does not form an orthonormal basis, but a frame.
In this case, the $N$-term approximation using the $N$ largest
coefficients is considered which, in case of an orthonormal basis, is
the same as the best $N$-term approximation, but not in general.
The term `optimally sparse approximation' is then awarded to those
representation systems which deliver the fastest possible decay rate
in $N$ for all cartoon-like images, where we consider log-factors as
negligible, thereby providing an answer to (P2).

\subsection{Why is 3D the Crucial Dimension?}
\label{subsec:3D}

We already identified the step from 1D to 2D as crucial for the appearance of anisotropic features at all. Hence one
might ask: Is is sufficient to consider only the 2D situation, and higher dimensions can be treated similarly? Or:
Does each dimension causes its own problems? To answer these questions, let us consider the step from 2D to 3D which
shows a curious phenomenon. A 3D function can exhibit point (= 0D), curvilinear (= 1D), and surface (= 2D) singularities.
Thus, suddenly anisotropic features appear in two different dimensions: As one-dimensional and as two-dimensional
features. Hence, the 3D situation has to be analyzed with particular care. It is not at all clear whether two different
representation systems are required for optimally approximating both types of anisotropic features simultaneously, or
whether one system will suffice. This shows that the step from 2D to 3D can justifiably be also coined `crucial'.
Once it is known how to handle anisotropic features of different dimensions, the step from 3D to 4D can be dealt
with in a similar way as also the extension to even higher dimensions. Thus, answering (P3), we conclude that the
two crucial dimensions are 2D and 3D with higher dimensional situations deriving from the analysis of those.

\subsection{Performance of Shearlets and Other Directional Systems}

Within the framework we just briefly outlined, it can be shown that wavelets do not provide optimally sparse approximations
of cartoon-like images. This initiated a flurry of activity within the applied harmonic analysis community with the aim to
develop so-called {\em directional} representation systems which satisfy this benchmark, certainly besides other desirable
properties depending in the application at hand. In 2004, Cand\'{e}s and Donoho were the first to introduce with the
tight curvelet frames a directional representation system which provides provably optimally sparse approximations
of cartoon-like images in the sense we discussed. One year later, contourlets were introduced by Do and Vetterli \cite{DV05},
which similarly derived an optimal approximation rate. The first analysis of the performance of (band-limited) shearlet
frames was undertaken by Guo and Labate in \cite{GL07}, who proved that these shearlets also do satisfy this benchmark. In the
situation of (band-limited) shearlets the analysis was then driven even further, and very recently Guo and Labate proved a
similar result for 3D cartoon-like images which in this case are defined as a function which is $C^2$ apart from a
$C^2$ discontinuity surface, i.e., focusing on only one of the types of anisotropic features we are facing in 3D.

\subsection{Band-Limited Versus Compactly Supported Systems}

The results mentioned in the previous subsection only concerned band-limited systems. Even in the contourlet
case, although compactly supported contourlets seem to be included, the proof for optimal sparsity only works
for band-limited generators due to the requirement of infinite directional vanishing moments. However, for various
applications compactly supported generators are inevitable, wherefore already in the wavelet case the introduction
of compactly supported wavelets was a major advance. Prominent examples of such applications are imaging sciences, when
an image might need to be denoised while avoiding a smoothing of the edges, or in the theory of partial differential
equations as a generating system for a trial space in order to ensure fast computational realizations.

So far, shearlets are the only system, for which a theory for compactly supported generators has been developed
and compactly supported shearlet frames have been constructed \cite{KKL10a}, see also the survey paper \cite{KLL10_2}.
It should though be mentioned that these frames are somehow close to being tight, but at this point it is not clear
whether also compactly supported {\em tight} shearlet frames can be constructed. Interestingly, it was proved in
\cite{KL10} that this class of shearlet frames also delivers optimally sparse approximations of the 2D cartoon-like
image model class with a very different proof than \cite{GL07} now adapted to the particular nature of compactly
supported generators. And with \cite{KLL10} the 3D situation is now also fully understood, even taking the two
different types of anisotropic features -- curvilinear and surface singularities -- into account.

\subsection{Outline}

In Sect.~\ref{sec:cartoon-images}, we introduce the 2D and 3D cartoon-like image model class. Optimality of
sparse approximations of this class are then discussed in Sect.~\ref{sec:linear-non-linear}. Sect.~\ref{sec:pyram-adapt-shearl}
is concerned with the introduction of  3D shearlet systems with both band-limited and compactly supported generators,
which are shown to provide optimally sparse approximations within this class  in the final Sect.~\ref{sec:optim-sparse-appr}.

\section{Cartoon-like Image Class}
\label{sec:cartoon-images}

We start by making the in the introduction of this chapter already intuitively derived definition of cartoon-like
images mathematically precise. We start with the most basic definition of this class which was also historically
first stated in \cite{Don99}. We allow ourselves to state this together with its 3D version from \cite{GL10_3d}
by remarking that $d$ could be either $d=2$ or $d=3$.

For fixed $\mu > 0$, the {\em class $\cE^2(\R^d)$ of cartoon-like image} shall be the set of
functions $f:\mathbb{R}^d \to \mathbb{C}$ of the form
    \begin{equation*}
  f = f_0 + f_1 \chi_{B},
\end{equation*}
where $B \subset \itvcc{0}{1}^d$ and $f_i \in C^2(\mathbb{R}^d)$ with $\supp{ f_0} \subset \itvcc{0}{1}^d$ and
$\norm[C^2]{f_i} \leq \mu$ for each $i=0,1$. For dimension $d=2$, we assume that $\partial B$ is a closed $C^2$-curve
with curvature bounded by $\nu$, and, for $d=3$, the discontinuity $\partial B$ shall be a closed $C^2$-surface with
principal curvatures bounded by $\nu$.
An indiscriminately chosen cartoon-like function $f = \chi_B$, where the discontinuity surface $\pa B$ is a deformed sphere in $\R^3$,
is depicted in Fig.~\ref{fig:cartoon}.
\begin{figure}[ht!]
\centering
\includegraphics{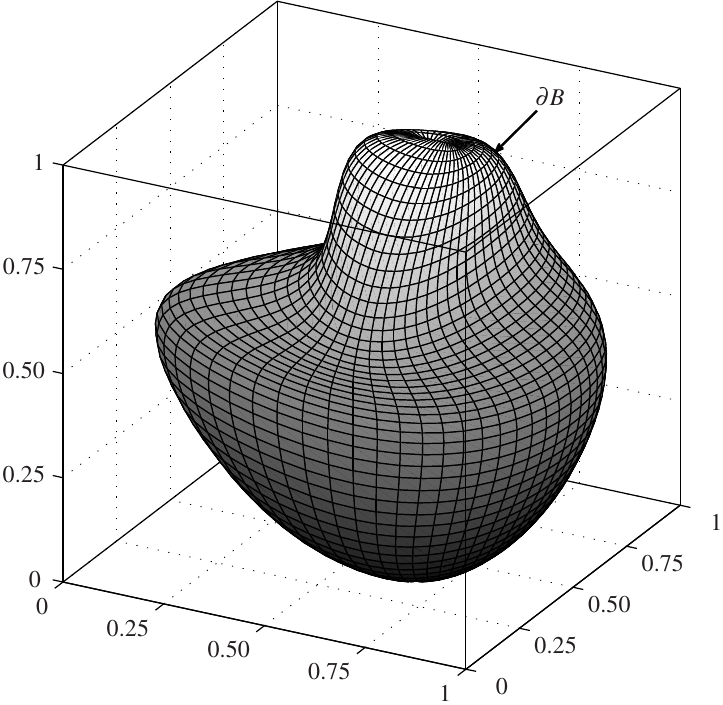}
\caption{A simple cartoon-like image $f = \chi_B \in \cE^2_L(\R^3)$
  with $L=1$ for dimension $d=3$, where the
  discontinuity surface $\pa B$ is a deformed sphere.} \label{fig:cartoon}
\end{figure}

Since `objects' in images often have sharp corners, in \cite{CD04}
for 2D and in \cite{KLL10} for 3D also less regular images were
allowed, where $\pa B$ is only assumed to be \emph{piecewise}
$C^2$-smooth. We note that this viewpoint is also essential for being
able to analyze the behavior of a system with respect to the two
different types of anisotropic features appearing in 3D; see the
discussion in Subsection \ref{subsec:3D}. Letting $L \in \N$ denote the number
of $C^2$ pieces, we speak of the extended class of {\em cartoon-like
  images $\cE^2_L(\R^d)$} as consisting of cartoon-like images having
$C^2$-smoothness apart from a piecewise $C^2$ discontinuity curve in
the 2D setting and a piecewise $C^2$ discontinuity surface in the 3D
setting. Indeed, in the 3D setting, besides the $C^2$ discontinuity
surfaces, this model exhibits curvilinear $C^2$ singularities as well
as point singularities, \eg the cartoon-like image $f = \chi_B$ in
Fig.~\ref{fig:cartoon-piecewise} exhibits a discontinuity surface $\pa
B \subset \R^3$ consisting of \emph{three} $C^2$-smooth surfaces with
point and curvilinear singularities where these surfaces meet.

\begin{figure}[ht]
\centering
\includegraphics{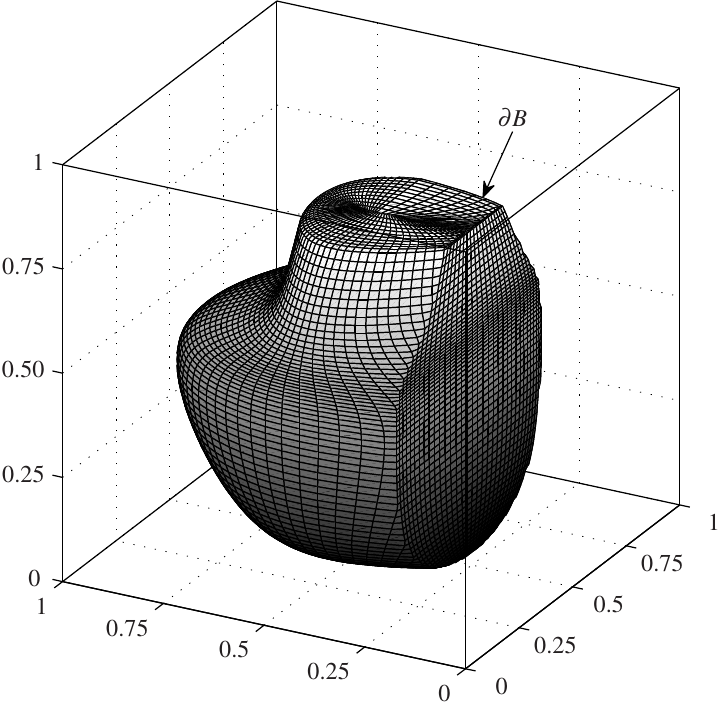}
\caption{A cartoon-like image $f = \chi_B \in \cE_L^2(\R^3)$ with $L=3$, where the
  discontinuity surface $\pa B$ is piecewise $C^2$ smooth.}
 \label{fig:cartoon-piecewise}
\end{figure}

The model in \cite{KLL10} goes even one step further and considers a
different regularity for the smooth parts, say being in $C^\beta$,
and for the smooth pieces of the discontinuity, say being in
$C^\alpha$ with $1 \le \alpha \le \beta \le 2$. This very general
class of cartoon-like images is then denoted by
$\cE^\beta_{\alpha,L}(\R^d)$, with the agreement that $\cE^2_L(\R^d)
= \cE^\beta_{\alpha,L}(\R^d)$ for $\alpha=\beta=2$.

For the purpose of clarity, in the sequel we will focus on the first
most basic cartoon-like model where $\alpha=\beta=2$, and add hints on
generalizations when appropriate (in particular, in Sect.~\ref{sec:some-extensions}).

\section{Sparse Approximations}
\label{sec:linear-non-linear}

After having clarified the model situation, 
we will now discuss which measure for the accuracy of approximation by
representation systems we choose, and what optimality means in this
case.

\subsection{(Non-Linear) $N$-term Approximations}
\label{sec:non-line-appr}

Let $\cC$ denote a given class of elements in a separable Hilbert
space $\cH$ with norm $\norm{\cdot}=\innerprod{\cdot}{\cdot}^{1/2}$ and $\Phi=(\phi_i)_{i \in I}$ a dictionary for $\cH$, \ie
$\overline{\Span}\Phi=\cH$, with indexing set $I$. The dictionary
$\Phi$ plays the role of our representation system.
Later $\cC$ will be chosen to be the class of cartoon-like images and
$\Phi$ a shearlet frame, but for now we will assume this more general
setting. We now seek to approximate each single element of $\cC$ with
elements from $\Phi$ by `few' terms of this system. Approximation
theory provides us with the concept of best $N$-term approximation
which we now introduce; for a general introduction to approximation theory, we
refer to \cite{ConApp}.

For this, let $f \in \cC$ be arbitrarily chosen. Since $\Phi$ is a
complete system, for any $\eps >0 $ there exists a finite linear
combination of elements from $\Phi$ of the form
\[g = \sum_{i \in F} c_i \phi_i \quad \text{with $F \subset I$ finite,
  \ie $\card{F}<\infty$}
\]
such that $\norm{f-g} \le \eps$. Moreover, if $\Phi$ is a frame
with countable indexing set $I$, there exists a sequence $(c_i)_{i \in
  I} \in \ell_2(I)$ such that the representation
\[
f = \sum_{i \in I} c_i \phi_i  
\]
holds with convergence in the Hilbert space norm $\norm{\cdot}$.
The reader should notice that, if $\Phi$ does not form a basis, this
representation of $f$ is certainly not the only possible one. Letting
now $N \in \NN$, we aim to approximate $f$ by only $N$ terms of
$\Phi$, i.e., by
\[
\sum_{i \in {I}_N} {c}_i \phi_i \quad \text{with } {I}_N \subset I, \, \card{{I}_N} = N, 
\]
which is termed {\em $N$-term approximation} to $f$. This approximation is typically
{\em non-linear} in the sense that if $f_N$ is an $N$-term
approximation to $f$ with indices $I_N$ and $g_N$ is an $N$-term
approximation to some $g \in \cC$ with indices $J_N$, then $f_N + g_N$
is only an $N$-term approximation to $f+g$ in case $I_N=J_N$.

But certainly we would like to pick the `best' approximation with the
accuracy of approximation measured in the Hilbert space norm. We
define the {\em best $N$-term approximation} to $f$ by the $N$-term
approximation
\[
f_N = \sum_{i \in {I}_N} {c}_i \phi_i,
\]
which satisfies that, for all $I_N \subset I$, $\card{I_N} = N$, and
for all scalars $(c_i)_{i \in I}$, 
\[
\norm{f-f_N} \le \normBig{f - \sum_{i \in I_N} c_i \phi_i}.
\]

Let us next discuss the notion of best $N$-term approximation for the
special cases of $\Phi$ forming an orthornomal basis, a tight frame,
and a general frame alongside an error estimate for the accuracy of
this approximation.

\subsubsection{Orthonormal Bases}
\label{subsubsec:ONB}

Let $\Phi$ be an orthonormal basis for $\cH$. In this case, we can
actually write down the best $N$-term approximation $f_N = \sum_{i \in
  {I}_N} {c}_i \phi_i$ for $f$. Since in this case
\[
f = \sum_{i \in I} \ip{f}{\phi_i} \phi_i,
\]
and this representation is unique, we obtain
\begin{align*}
  \|f-f_N\|_\cH &=  \Big\|\sum_{i \in I} \ip{f}{\phi_i} \phi_i - \sum_{i \in {I}_N} {c}_i \phi_i\Big\|\\
  &=  \Big\|\sum_{i \in {I}_N} [\ip{f}{\phi_i} -{c}_i] \phi_i + \sum_{i \in I \setminus {I}_N} \ip{f}{\phi_i} \phi_i\Big\|\\
  &= \|(\ip{f}{\phi_i} -{c}_i)_{i \in {I}_N}\|_{\ell^2} + \|(
  \ip{f}{\phi_i})_{i \in I \setminus {I}_N}\|_{\ell^2}.
\end{align*}
The first term $\|(\ip{f}{\phi_i} -{c}_i)_{i \in {I}_N}\|_{\ell^2}$
can be minimized by choosing ${c}_i = \ip{f}{\phi_i}$ for all $i \in
{I}_N$. And the second term $\|( \ip{f}{\phi_i})_{i \in I \setminus
  {I}_N}\|_{\ell^2}$ can be minimized by choosing ${I}_N$ to be the indices
of the $N$ largest coefficients $\ip{f}{\phi_i}$ in magnitude. Notice
that this does not uniquely determine $f_N$ since some coefficients
$\ip{f}{\phi_i}$ might have the same magnitude. But it characterizes
the set of best $N$-term approximations to some $f \in \cC$ precisely.
Even more, we have complete control of the error of best $N$-term
approximation by
\begin{equation}
 \label{eq:errorONB} \|f-f_N\| = \|( \ip{f}{\phi_i})_{i \in I
  \setminus {I}_N}\|_{\ell^2}.
\end{equation}

\subsubsection{Tight Frames}
\label{sec:tight-frames}

Assume now that $\Phi$ constitutes a tight frame with bound $A=1$ for $\cH$. In this
situation, we still have
\[
f = \sum_{i \in I} \ip{f}{\phi_i} \phi_i,
\]
but this expansion is now not unique anymore. Moreover, the frame
elements are not orthogonal. Both conditions prohibit an analysis of
the error of best $N$-term approximation as in the previously
considered situation of an orthonormal basis. And in fact, examples can be provided
to show that selecting the $N$ largest coefficients $\ip{f}{\phi_i}$
in magnitude does not always lead to the {\em best} $N$-term
approximation, but merely to {\em an} $N$-term approximation. To
be able to still analyze the approximation error, one typically -- as will
be also our choice in the sequel -- chooses the $N$-term approximation
provided by the indices ${I}_N$ associated with the $N$ largest
coefficients $\ip{f}{\phi_i}$ in magnitude with these coefficients,
i.e.,
\[
f_N = \sum_{i \in {I}_N} \ip{f}{\phi_i} \phi_i.
\]
This selection also allows for some control of the approximation in
the Hilbert space norm, which we will defer to the next subsection
in which we consider the more general case of arbitrary frames.

\subsubsection{General Frames}
\label{sec:general-frames}

Let now $\Phi$ form a frame for $\cH$ with frame bounds $A$ and $B$,
and let $(\tilde\phi_i)_{i \in I}$ denote the canonical dual frame. We
then consider the expansion of $f$ in terms of this dual frame, i.e.,
\begin{equation}\label{eq:first} f = \sum_{i \in I} \ip{f}{\phi_i}
  \tilde{\phi}_i.
\end{equation} Notice that we could also consider
\[
f = \sum_{i \in I} \ip{f}{\tilde{\phi}_i} \phi_i.
\]
Let us explain, why the first form is of more interest to us in this
chapter. By definition, we have $(\ip{f}{\tilde{\phi}_i})_{i \in I} \in \ell^2(I)$ as well
as $(\ip{f}{\phi_i})_{i \in I} \in \ell^2(I)$. Since we only consider expansions
of functions $f$ belonging to a subset $\cC$ of $\cH$, this can, at
least, potentially improve the decay rate of the coefficients so that
they belong to $\ell^p(I)$ for some $p<2$. This is exactly what is
understood by {\em sparse approximation} (also called {\em
  compressible approximations} in the context of inverse problems). We
hence aim to analyze shearlets with respect to this behavior, i.e., the decay
rate of shearlet coefficients. This then naturally leads to form
\eqref{eq:first}. We remark that in case of a tight frame, there is no
distinction necessary, since then $\tilde \phi_i = \phi_i$ for all $i
\in I$.

As in the tight frame case, it is not possible to derive a usable, explicit form
for the best $N$-term approximation. We therefore again crudely
approximate the best $N$-term approximation by choosing the $N$-term
approximation provided by the indices ${I}_N$ associated with the $N$
largest coefficients $\ip{f}{\phi_i}$ in magnitude with these
coefficients, i.e.,
\[
f_N = \sum_{i \in {I}_N} \ip{f}{\phi_i} \tilde \phi_i.
\]
But, surprisingly, even with this rather crude greedy selection procedure, we obtain very strong results for the approximation
rate of shearlets as we will see in Sect.~\ref{sec:optim-sparse-appr}.

The following result shows how the $N$-term approximation error can be bounded by the tail of the square of
the coefficients $c_i$. The reader might want to compare this result with the error in case of an orthonormal basis
stated in \eqref{eq:errorONB}.

\begin{lemma}\label{lemma:n-term-frame-approx}
Let $(\phi_i)_{i \in I}$ be a frame for $\cH$ with frame bounds $A$ and $B$, and let $(\tilde\phi_i)_{i \in I}$
be the canonical dual frame. Let $I_N \subset I$ with $\card{I_N} = N$, and let $f_N$ be the $N$-term approximation $f_N = \sum_{i \in I_N}
\innerprod{f}{\phi_i} \tilde \phi_i$. Then
\begin{equation}
\norm{f-f_N}^2 \le \frac1A \sum_{i \notin I_N} \abs{\innerprod{f}{\phi_i}}^2.\label{eq:n-term-frame-approx-bound}
\end{equation}
\end{lemma}
\begin{proof}
  Recall that the canonical dual frame satisfies the frame inequality
  with bounds $B^{-1}$ and $A^{-1}$. At first hand, it therefore might look
  as if the estimate~(\ref{eq:n-term-frame-approx-bound}) should
  follow directly from the frame inequality for the canonical dual.
  However, since the sum in~(\ref{eq:n-term-frame-approx-bound}) does not run over the
  entire index set $i \in I$, but only $I \setminus I_N$, this is not
  the case. So, to prove the lemma, we first consider
\begin{align}
\norm{f-f_N} &= \sup\setprop{\abs{\innerprod{f-f_N}{g}}}{g \in \cH, \norm{g}=1} \nonumber \\
&= \sup\setprop{\absBig{\sum_{i\notin I_N}\innerprod{f}{\phi_i}\innerprod{\tilde\phi_i}{g}}}{g \in \cH, \norm{g}=1}. \label{eq:reisz-estimate}
\end{align}
Using Cauchy-Schwarz' inequality, we then have that
\begin{align*}
\absBig{\sum_{i \notin I_N}\innerprod{f}{\phi_i}\innerprod{\tilde\phi_i}{g}}^2
\le \sum_{i \notin I_N}\abs{\innerprod{f}{\phi_i}}^2
\sum_{i \notin I_N}\abs{\innerprod{\tilde\phi_i}{g}}^2 \le A^{-1} \norm{g}^2
\sum_{i \notin I_N}\abs{\innerprod{f}{\phi_i}}^2,
\end{align*}
where we have used the upper frame inequality for the dual frame $(\tilde \phi_i)_i$ in the second step. We can now continue
(\ref{eq:reisz-estimate}) and arrive at
\begin{align*}
\norm{f-f_N}^2 \le \sup\setprop{\frac1A \norm{g}^2 \sum_{i \notin I_N}\abs{\innerprod{f}{\phi_i}}^2}{g \in \cH, \norm{g}=1}
= \frac1A \sum_{i \notin I_N} \abs{\innerprod{f}{\phi_i}}^2.
\end{align*}
\end{proof}

Relating to the previous discussion about the decay of coefficients
$\ip{f}{\phi_i}$,
let $c^\ast$ denote the non-increasing (in modulus) rearrangement of
$c=(c_i)_{i \in I}=(\innerprod{f}{\phi_i})_{i \in I}$, \eg
$c^\ast_{\,n}$ denotes the $n$th largest coefficient of $c$ in
modulus. This rearrangement corresponds to a bijection $\pi:\N \to I$
that satisfies 
\begin{equation*}
\pi:\N \to I, \quad c_{\pi(n)}=c^\ast_{\,n} \text{ for all $n \in
  \N$}. 
 \end{equation*}
Strictly speaking, the rearrangement (and hence the mapping $\pi$) might
not be unique; we will simply take $c^\ast$ to be one of these rearrangements.
Since $c \in \ell^2(I)$, also $c^\ast \in \ell^2(\N)$. Suppose further
that $\abs{c^\ast_{n}}$ even decays as
\[
\abs{c^\ast_{n}} \lesssim n^{-(\alpha+1)/2} \qquad \text{for} \quad n \to \infty
\]
for some $\alpha>0$, where the notation $h(n)\lesssim g(n)$ means that
there exists a $C>0$ such that $h(n) \le C g(n)$, \ie $h(n)= O(g(n))$.
Clearly, we then have $c^\ast \in \ell^p(\N)$ for $p\ge \tfrac{2}{\alpha+1}$.
By Lemma~\ref{lemma:n-term-frame-approx}, the $N$-term
approximation error will therefore decay as
\begin{equation*}
\norm{f-f_N}^2 \le \frac1A \sum_{n>N} \abs{c^\ast_{n}}^2 \lesssim \sum_{n>N} n^{-\alpha+1} \asymp N^{-\alpha},
\end{equation*}
where $f_N$ is the $N$-term approximation of $f$ by keeping the $N$ largest coefficients, that is,
\begin{equation}
f_N = \sum_{n=1}^N c^*_{\,n} \, \tilde{\phi}_{\pi(n)}.
\label{eq:frame-n-term-largest}
\end{equation}
The notation $h(n)\asymp g(n)$, also written $h(n)=\Theta(g(n))$, used above means that $h$ is bounded both above and below
by $g$ asymptotically as $n\to \infty$, that is, $h(n)=O(g(n))$ \emph{and} $g(n)=O(h(n))$.

\subsection{A Notion of Optimality}
\label{sec:optimal-sparsity}

We now return to the setting of functions spaces $\cH=L^2(\R^d)$,
where the subset $\cC$ will be the class of cartoon-like images, that
is, $\cC=\cE_L^2(\R^d)$. We then aim for a benchmark, i.e., an
optimality statement, for sparse approximation of functions in
$\cE_L^2(\R^d)$. For this, we will again only require that our
representation system $\Phi$ is a dictionary, that is, we assume only
that $\Phi = (\phi_i)_{i \in I}$ is a complete family of functions in
$L^2(\R^d)$ with $I$ not necessarily being countable. Without loss of
generality, we can assume that the elements $\phi_i$ are normalized,
\ie $\norm[L^2]{\phi_i}=1$ for all $i \in I$.
For $f \in \cE_L^2(\R^d)$ we then consider expansions of the form
\[
f = \sum_{i \in I_f } c_i \, \phi_i,
\]
where $I_f \subset I$ is a countable selection from $I$ that may
depend on $f$. Relating to the previous subsection, the first $N$
elements of $\Phi_f:=\{\phi_i\}_{i \in
I_f}$ could for instance be the $N$ terms from
$\Phi$ selected for the best $N$-term approximation of $f$.

Since artificial cases shall be avoided, this selection procedure has
the following natural restriction which is usually termed {\em
  polynomial depth search}: The $n$th term in $\Phi_f$ is obtained by
only searching through the first $q(n)$ elements of the list
$\Phi_f$, where $q$ is a polynomial. Moreover, the selection rule
may \emph{adaptively} depend on $f$, and the $n$th element may also be
modified adaptively and depend on the first $(n-1)$th chosen elements.
We shall denote any sequence of coefficients $c_i$ chosen according to
these restrictions by $c(f)=(c_i)_i$. The role of the polynomial $q$
is to limit how deep or how far down in the listed dictionary $\Phi_f$
we are allowed to search for the next element $\phi_i$ in the
approximation. Without such a depth search limit, one could choose
$\Phi$ to be a countable, dense subset of $L^2(\R^d)$ which would
yield arbitrarily good sparse approximations, but also infeasible
approximations in practise.

Using information theoretic arguments, it was then shown in
\cite{Don01,KLL10}, that almost no matter what selection procedure we
use to find the coefficients $c(f)$, we cannot have
$\norm[\ell^p]{c(f)}$ bounded for $p<\frac{2(d-1)}{d+1}$ for $d=2, 3$.

\begin{theorem}[\cite{Don01,KLL10}]
\label{theo:benchmark}
Retaining the definitions and notations in this subsection and allowing only polynomial depth search, we obtain
\[
\max_{f\in \cE_{L}^{2}(\R^d)} \norm[\ell^p]{c(f)} = +\infty,  \qquad \text{for } p<\frac{2(d-1)}{d+1}.
\]
\end{theorem}

In case $\Phi$ is an orthonormal basis for
$L^2(\RR^d)$,  the norm $\norm[\ell^p]{c(f)}$ is trivially bounded
for $p \ge 2$ since we can take $c(f)=(c_i)_{i \in I} =
(\innerprod{f}{\phi_i})_{i \in I}$. Although not explicitly stated,
the proof can be straightforwardly extended from 3D to higher
dimensions as also the definition of cartoon-like images can be
similarly extended. It is then intriguing to analyze the behavior of
$\frac{2(d-1)}{d+1}$ from Thm.~\ref{theo:benchmark}. In fact, as $d \to
\infty$, we observe that $\frac{2(d-1)}{d+1} \to 2$. Thus, the decay of
any ${c(f)}$ for cartoon-like images becomes slower as
$d$ grows and approaches $\ell^2$, which -- as we just mentioned -- is
actually the rate guaranteed for \emph{all} $f \in L^2(\R^d)$.

Thm.~\ref{theo:benchmark} is truly a statement about the optimal achievable sparsity level: No representation system
-- up to the restrictions described above -- can deliver approximations for $\cE^{2}_{L}(\R^d)$ with coefficients
satisfying $c(f) \in \ell_p$ for $p < \frac{2(d-1)}{d+1}$.
This implies the following lower bound 
\begin{equation}
  \label{eq:coeff-fastest-decay}
  c(f)^\ast_{\,n} \gtrsim n^{-\frac{d+1}{2(d-1)}} =
  \begin{cases}
    n^{-3/2} & : \quad d=2, \\
   n^{-1} & : \quad d=3.
  \end{cases}
\end{equation}
where $c(f)^\ast=(c(f)^\ast_{n})_{n\in \N}$ is a decreasing (in modulus) arrangement of the coefficients $c(f)$.

One might ask how this relates to the approximation error of (best) $N$-term
approximation discussed before. For simplicity, suppose for a moment
that $\Phi$ is actually an orthonormal basis (or more generally a
Riesz basis) for $L^2(\R^d)$ with $d=2$ and $d=3$. Then -- as
discussed in Sect.~\ref{subsubsec:ONB} -- the \emph{best} $N$-term
approximation to $f \in \cE^{2}_{L}(\R^d)$ is obtained by keeping the
$N$ largest coefficients. Using the error
estimate~\eqref{eq:errorONB} as well as (\ref{eq:coeff-fastest-decay}), we obtain
\begin{align*}
  \norm[L^2]{f-f_N}^2 =  \sum_{n>N} \abs{c(f)^\ast_{\,n}}^2 \gtrsim \sum_{n>N} n^{-\frac{d+1}{d-1}} \asymp N^{-\frac{2}{d-1}},
\end{align*}
i.e., the best $N$-term approximation error $\norm[L^2]{f-f_N}^2$
behaves asymptotically as $N^{-\frac{2}{d-1}}$ or worse. If, more
generally, $\Phi$ is a frame, and $f_N$ is chosen as in
(\ref{eq:frame-n-term-largest}), we can similarly conclude that the
asymptotic lower bound for $\norm[L^2]{f-f_N}^2$ is
$N^{-\frac{2}{d-1}}$, that is, the optimally achievable rate is, at
best, $N^{-\frac{2}{d-1}}$. Thus, this optimal rate can be used as a
benchmark for measuring the sparse approximation ability of
cartoon-like images of different representation systems. Let us phrase
this formally.

\begin{definition}
\label{def:optimal}
Let $\Phi = (\phi_i)_{i \in I}$ be a frame for $L^2(\R^d)$ with $d=2$
or $d=3$. We say that \emph{$\Phi$ provides optimally sparse
  approximations} of cartoon-like images 
if, for each $f \in \cE^{2}_{L}(\R^d)$, the
associated $N$-term approximation $f_N$ (cf.
\eqref{eq:frame-n-term-largest}) by keeping the $N$ largest
coefficients of $c=c(f)=(\innerprod{f}{\phi_i})_{i\in I}$  satisfies
\begin{align}
  \norm[L^2]{f-f_N}^2 &\lesssim N^{-\frac{2}{d-1}} \qquad \text{as $N \to
    \infty$,} \label{eq:approx-fastest-decay-upper}\\
\intertext{and} \abs{c^\ast_{\,n}} &\lesssim
  n^{-\frac{d+1}{2(d-1)}} \qquad \text{as $n \to
    \infty$,} \label{eq:coeff-fastest-decay-upper}
\end{align}
where we ignore $\log$-factors. 
\end{definition}

Note that, for frames $\Phi$, the bound $\abs{c^\ast_{\,n}} \lesssim
n^{-\frac{d+1}{2(d-1)}}$
automatically implies that $ \norm{f-f_N}^2 \lesssim N^{-\frac{2}{d-1}}$ whenever
$f_N$ is chosen as in Eqn.~\eqref{eq:frame-n-term-largest}. This follows
from Lemma~\ref{lemma:n-term-frame-approx} and the estimate
\begin{align}
  \sum_{n>N} \abs{c^\ast_{\,n}}^2 \lesssim \sum_{n>N}
  n^{-\frac{d+1}{d-1}} \lesssim \int_N^\infty x^{-\frac{d+1}{d-1}}
  \D x \le C \cdot N^{-\frac{2}{d-1}}, \label{eq:sparsity-implies-approx}
\end{align}
where we have used that $-\frac{d+1}{d-1}+1=-\frac{2}{d-1}$. Hence, we are
searching for a representation system $\Phi$ which forms a frame and
delivers decay of $c=(\innerprod{f}{\phi_i})_{i\in I}$ as (up to log-factors)
\begin{align}
 \abs{c^*_{\, n}} \lesssim n^{-\frac{d+1}{2(d-1)}} = \begin{cases}
    n^{-3/2} & : \quad d=2, \\
   n^{-1} & : \quad d=3.
  \end{cases}
\label{eq:sought-sparsity}
\end{align}
as $n \to \infty$ for any cartoon-like image.

\subsection{Approximation by Fourier Series and Wavelets}
\label{sec:some-illustr-exampl}

We will next study two examples of more traditional representation
systems --  the Fourier basis and wavelets -- with respect to
their ability to meet this benchmark. For this, we choose the function
$f = \chi_B$, where $B$ is a ball contained in $\itvcc{0}{1}^d$, again
$d=2$ or $d=3$, as a simple cartoon-like image in $\cE^2_L(\R^d)$ with
$L=1$, analyze the error $\norm{f-f_N}^2$ for $f_N$ being the $N$-term
approximation by the $N$ largest coefficients and compare with the
optimal decay rate stated in Definition \ref{def:optimal}. It will
however turn out that these systems are far from providing optimally sparse
approximations of cartoon-like images, thus underlining the pressing
need to introduce representation systems delivering this optimal rate;
and we already now refer to Sect.~\ref{sec:optim-sparse-appr} in which
shearlets will be proven to satisfy this property.

Since Fourier series and wavelet systems are orthonormal bases (or
more generally, Riesz bases) the best $N$-term approximation is found
by keeping the $N$ largest coefficients as discussed in Sect.~\ref{subsubsec:ONB}.

\subsubsection{Fourier Series}
\label{sec:fourier-series}
  The error of the best $N$-term Fourier series approximation
  of a typical cartoon-like image decays asymptotically as
  $N^{-1/d}$. The following proposition shows this behavior in the
  case of a very simple cartoon-like image: The characteristic function on a
  ball.

\begin{proposition}\label{prop:fourier-series-ball-decay}
  Let $d\in \N$, and let $\Phi=(\mathrm{e}^{2\pi ikx})_{k \in \Z^d}$. Suppose $f=\chi_B$, where $B$ is a ball contained in
  $\itvcc{0}{1}^d$. Then
  \[ \norm[L^2]{f-f_n}^2 \asymp N^{-1/d} \qquad \text{for } N\to
  \infty,\] where $f_N$ is the best $N$-term approximation from
  $\Phi$.
\end{proposition}

\begin{proof}
  We fix a new origin as the center of the ball $B$. Then $f$ is a
  radial function $f(x)=h(\norm[2]{x})$ for $x\in \R^d$. The Fourier
  transform of $f$ is also a radial function and can expressed
  explicitly by Bessel functions of first kind
  \cite{MR1237988,MR2587580}:
  \[ \hat f(\xi) = r^{d/2} \frac{J_{d/2}(2\pi
    r\norm[2]{\xi})}{\norm[2]{\xi}^{d/2}},
  \]
  where $r$ is the radius of the ball $B$. Since the Bessel function
  $J_{d/2}(x)$ decays like
  $x^{-1/2}$ as $x \to \infty$, the Fourier transform of $f$ decays
  like $\abssmall{\hat f(\xi)} \asymp \norm[2]{\xi}^{-(d+1)/2}$ as
  $\norm[2]{\xi}\to \infty$. Letting $I_N = \{k \in \Z^d:
  \norm[2]{k} \le N\}$ and $f_{I_N}$ be the partial Fourier sum
  with terms from $I_N$, we obtain
  \begin{align*}
    \norm[L^2]{f-f_{I_N}}^2 &= \sum_{k \not \in I_N} \abs{\hat f(k)}^2
    \asymp 
    \int_{\norm[2]{\xi} > N} \norm[2]{\xi}^{-(d+1)} \D\xi \\ &=
    \int_N^\infty r^{-(d+1)} r^{(d-1)} \D r = \int_N^\infty r^{-2} \D r =
    N^{-1}.
  \end{align*}
  The conclusion now follows from the cardinality of $\card{I_N}\asymp
  N^d$ as $N \to \infty$.
\end{proof}

\subsubsection{Wavelets}
\label{sec:wavelets}

Since wavelets are designed to deliver sparse representations of
singularities -- see Chapter \cite{Intro} -- we expect this system to
outperform the Fourier approach. This will indeed be the case.
However, the optimal rate will still by far be missed. The best
$N$-term approximation of a typical cartoon-like image using a wavelet
basis performs only slightly better than Fourier series with asymptotic behavior
as $N^{-1/(d-1)}$. This is illustrated by the following result.

\begin{proposition}\label{prop:wavelets-ball-decay}
  Let $d = 2,3$, and let $\Phi$ be a wavelet basis for $L^2(\R^d)$ or
  $L^2(\itvcc{0}{1}^d)$. 
  Suppose $f=\chi_B$, where $B$ is a ball contained in
  $\itvcc{0}{1}^d$. Then
\[ \norm[L^2]{f-f_n}^2 \asymp N^{-\frac{1}{d-1}} \qquad \text{for } N\to \infty,\]
where $f_N$ is the best $N$-term approximation from $\Phi$.
\end{proposition}

\begin{proof}
Let us first consider wavelet approximation by the Haar
tensor wavelet basis for $L^2(\itvcc{0}{1}^d)$ of the form
\[\setprop{\phi_{0,k}}{ \abs{k}
\le 2^J-1} \cup \setprop{\psi^1_{j,k},\dots
\psi^{2d-1}_{j,k}}{j\ge J, \abs{k} \le 2^{j-J}-1},
\]
where $J \in \N$, $k\in \N_0^d$, and $g_{j,k}=2^{jd/2}g(2^j\cdot -k)$
for $g \in L^2(\R^d)$.
There are only a finite number of coefficients of the form
$\innerprods{f}{\phi_{0,k}}$, hence we do not need to consider these
for our asymptotic estimate. For simplicity, we take $J=0$. At scale
$j\ge 0$ there exist $\Theta(2^{j(d-1)})$ non-zero wavelet coefficients,
since the surface area of $\pa B$ is finite and the wavelet elements
are of size $2^{-j} \times \dots \times 2^{-j}$.

To illustrate the calculations leading to the sought approximation error rate, we
will first consider the case where $B$ is a cube in $\itvcc{0}{1}^d$.
For this, we first consider the non-zero coefficients associated with the
face of the cube containing the point $(b,c,\dots,c)$. For scale $j$,
let $k$ be such that $\supp \psi^1_{j,k} \cap \supp f \neq \emptyset$,
where $\psi^1(x)=h(x_1)p(x_2)\cdots p(x_d)$ and $h$ and $p$ are the
Haar wavelet and scaling function, respectively. Assume that $b$ is
located in the first half of the interval
$\itvcc{2^{-j}k_1}{2^{-j}(k_1+1)}$; the other case can be handled
similarly. Then
\[
\abssmall{\innerprods{f}{\psi^1_{j,k}}} = \int_{2^{-j}k_1}^b
2^{jd/2} \D x_1 \prod_{i=2}^d \int_{2^{-j}k_i}^{2^{-j}(k_i+1)} \D x_i
= 
(b-2^{-j}k_1) 2^{-j(d-1)} 2^{jd/2} \asymp 2^{-jd/2},
\] where we have
used that $(b-2^{-j}k_1)$ will typically be of size $\tfrac{1}{4}\, 2^{-j}$.
 Note that for the chosen $j$ and $k$
above, we also have that $\innerprods{f}{\psi^l_{j,k}}=0$ for all
$l=2, \dots, 2^d-1$.

There will be $2 \cdot\ceilsmall{2c 2^{j(d-1)}}$ nonzero
coefficients of size $2^{-jd/2}$ associated with the wavelet $\psi^1$
at scale $j$. The same conclusion holds for the other wavelets
$\psi^l$, $l=2,\dots,2^d-1$. To summarize, at scale $j$ there will be
$C\, 2^{j(d-1)}$ nonzero coefficients of size $C\, 2^{-jd/2}$. On the
first $j_0$ scales, that is $j=0,1,\dots j_0$, we therefore have
$\sum_{j=0}^{j_0} 2^{j(d-1)} \asymp 2^{j_0(d-1)}$ nonzero
coefficients. The $n$th largest coefficient $c^\ast_{\,n}$ is of size
$n^{-\frac{d}{2(d-1)}} $ since, for $n=2^{j(d-1)}$, we have
\[
2^{-j\frac{d}{2}} = n^{-\frac{d}{2(d-1)}}.
\]
Therefore,
  \[ \norm[L^2]{f-f_{N}}^2= \sum_{n>N} \abs{c^\ast_{\,n}}^2  \asymp
  \sum_{n>N} n^{-\frac{d}{d-1}} \asymp \int_N^\infty x^{-\frac{d}{d-1}} \D x= \frac{d}{d-1}  N^{-\frac{1}{d-1}}.
  \]
Hence, for the  best $N$-term approximation $f_N$ of $f$ using a wavelet
basis, we obtain the asymptotic estimates
\begin{align*}
  \norm[L^2]{f-f_{N}}^2   = \Theta(N^{-\frac{1}{d-1}}) = \left\{\begin{aligned}
      \Theta(N^{-1}), &\qquad    \text{if } d=2,\\
      \Theta(N^{-1/2}), &\qquad  \text{if } d=3,
    \end{aligned}
  \right. .
\end{align*}

Let us now consider the situation that $B$ is a ball. In fact, in this case
we can do similar (but less transparent) calculations
leading to the same asymptotic estimates as above. We will not repeat
these calculations here, but simply remark that the upper asymptotic bound
in $\abssmall{\innerprods{f}{\psi^l_{j,k}}} \asymp 2^{-jd/2}$ can be
seen by the following general argument:
\[ \abssmall{\innerprods{f}{\psi^l_{j,k}}} \le \norm[L^\infty]{f}
\normsmall[L^1]{\psi^l_{j,k}} \le \normsmall[L^\infty]{f}
\normsmall[L^1]{\psi^l} 2^{-jd/2} \le C\, 2^{-jd/2}, \] which holds
for each $l=1,\dots,2^d-1$.

Finally, we can conclude from our calculations that choosing another
wavelet basis will not improve the approximation rate. 
\end{proof}


\begin{remark}
  We end this subsection with a remark on {\em linear} approximations. For a
  linear wavelet approximation of $f$ one would use
  \[
  f \approx \innerprod{f}{\phi_{0,0}} \phi_{0,0} +
  \sum_{l=1}^{2^d-1}\sum_{j=0}^{j_0}\sum_{\abs{k}\le 2^{j}-1}
  \innerprods{f}{\psi^l_{j,k}} \psi^l_{j,k}
  \]
  for some $j_0 >0$. If restricting to linear approximations, the summation
  order is not allowed to be changed, and we therefore need to include all
  coefficients from the first $j_0$ scales. At scale $j\ge 0$, there exist
  a total of $2^{jd}$ coefficients, which by our previous considerations can be bounded by $C\cdot
  2^{-jd/2}$. Hence, we include $2^j$ times as many coefficients
  as in the non-linear approximation on each
  scale.
  This implies that the error rate of the linear $N$-term
  wavelet approximation is $N^{-1/d}$, which is the {\em same} rate as
  obtained by Fourier approximations.
\end{remark}

\subsubsection{Key Problem}
\label{sec:key-problem}

The key problem of the suboptimal behavior of Fourier series and
wavelet bases is the fact that these systems are not generated by
anisotropic elements. Let us illustrate this for 2D in the case of
wavelets. Wavelet elements are isotropic due to the scaling matrix
$\diag{2^j, 2^j}$. However, already intuitively, approximating a curve
with isotropic elements requires many more elements than if the
analyzing elements would be anisotropic themselves, see
Fig.~\ref{fig:wavelet-curve} and~\ref{fig:shearlet-curve}.
\begin{figure}[ht]
  \centering
\parbox{0.42\textwidth}{%
\includegraphics[width=0.35\textwidth]{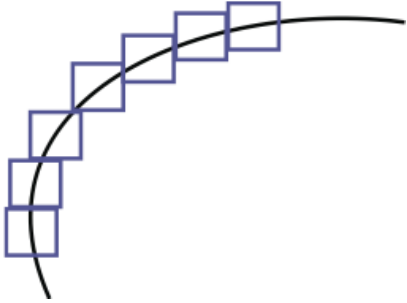}
\caption{Isotropic  elements capturing a dis\-con\-tinuity curve.}
 \label{fig:wavelet-curve}
 }%
\qquad 
\parbox{0.42\textwidth}{%
    \includegraphics[width=0.37\textwidth]{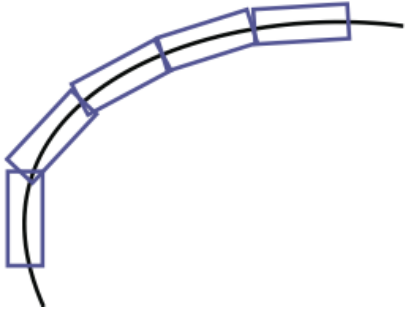}
\caption{Rotated, anisotropic elements capturing a dis\-con\-tinuity curve.}
 \label{fig:shearlet-curve}}
\end{figure}

Considering wavelets with anisotropic scaling will not remedy the
situation, since within one fixed scale one cannot control the
direction of the (now anisotropically shaped) elements. Thus, to
capture a discontinuity curve as in Fig.~\ref{fig:shearlet-curve}, one
needs not only anisotropic elements, but also a location parameter to
locate the elements on the curve and a rotation parameter to align the
elongated elements in the direction of the curve.

Let us finally remark why a parabolic scaling matrix $\diag{2^j, 2^{j/2}}$
will be natural to use as anisotropic scaling. Since the discontinuity curves of cartoon-like images are
$C^2$-smooth with bounded curvature, we may write the curve locally by a Taylor expansion. Let's assume
it has the form $(s,E(s))$ with
 \begin{equation*}
  E(s) = E(s') + E'(s') s + E''(t) s^2
\end{equation*}
near $s=s'$ for some $\abs{t} \in \itvcc{s'}{s}$. Clearly, the translation parameter will
be used to position the anisotropic element near $(s',E(s'))$, and the
orientation parameter to align with $(1,E'(s') s)$. If the length of the
element is $l$, then, due to the term $E''(t) s^2$, the most beneficial
height would be $l^2$. And, in fact, parabolic scaling yields precisely
this relation, i.e.,
\[
height \approx length^2.
\]

Hence, the main idea in the following will be to design a system which consists of
anisotropically shaped elements together with a directional parameter
to achieve the optimal approximation rate for cartoon-like images.

\section{Pyramid-Adapted Shearlet Systems}
\label{sec:pyram-adapt-shearl}

After we have set our benchmark for directional representation systems
in the sense of stating an optimality criteria for sparse
approximations of the cartoon-like image class $\cE^2_L(\R^d)$, we
next introduce classes of shearlet systems we claim behave optimally.
As already mentioned in the introduction of this chapter, optimally
sparse approximations were proven for a class of band-limited as well
as of compactly supported shearlet frames. For the definition of
cone-adapted discrete shearlets and, in particular, classes of
band-limited as well as of compactly supported shearlet frames leading
to optimally sparse approximations, we refer to Chapter \cite{Intro}. In this
section, we present the definition of discrete shearlets in 3D, from
which the mentioned definitions in the 2D situation can also be
directly concluded. As special cases, we then introduce particular
classes of band-limited as well as of compactly supported shearlet
frames, which will be shown to provide optimally approximations of
$\mathcal{E}_L^2(\R^3)$ and, with a slight modification which we will
elaborate on in Sect.~\ref{sec:some-extensions}, also for
$\mathcal{E}_{\alpha,L}^\beta(\R^3)$ with $1<\alpha \le \beta \le 2$.

\subsection{General Definition}
\label{sec:general-definition}

The first step in the definition of cone-adapted discrete 2D shearlets was a partitioning of 2D frequency domain into two
pairs of high-frequency cones and one low-frequency rectangle. We mimic this step by partitioning 3D frequency domain into
the three pairs of {\em pyramids} given by
\begin{align*}
  \mathcal{P} &=
      \{(\xi_1,\xi_2,\xi_3) \in \mathbb{R}^3 : |\xi_1| \ge 1,\,
      |\xi_2/\xi_1| \le 1,\, |\xi_3/\xi_1| \le 1\}, \\
 \tilde{\mathcal{P}}  &=   \{(\xi_1,\xi_2,\xi_3) \in \mathbb{R}^3 : |\xi_2| \ge 1,\, |\xi_1/\xi_2|
      \le 1,\,
      |\xi_3/\xi_2| \le 1\}, \\
 \breve{\mathcal{P}}   &=  \{(\xi_1,\xi_2,\xi_3) \in \mathbb{R}^3 : |\xi_3| \ge 1,\, |\xi_1/\xi_3|
      \le 1,\,
      |\xi_2/\xi_3| \le 1\},
 \end{align*}
and the centered cube
\[
\mathcal{C} = \{(\xi_1,\xi_2,\xi_3) \in \mathbb{R}^3 :
\norm[\infty]{(\xi_1,\xi_2, \xi_3)} < 1\}.
\]
This partition is illustrated in Fig.~\ref{fig:pyramids} which depicts the three pairs of pyramids
\begin{figure}[ht]
 \vspace*{-.5em}
\centering
\includegraphics{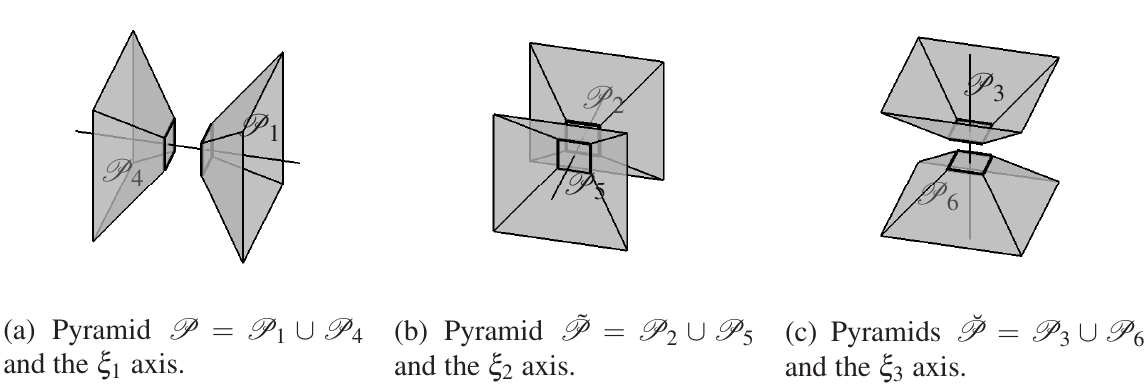}
\caption{The partition of the
  frequency domain: The `top' of the six pyramids.}
\label{fig:pyramids}
\end{figure}
and Fig.~\ref{fig:partition} depicting the centered cube surrounded by the three pairs of pyramids $\mathcal{P}$,
$\tilde{\mathcal{P}}$, and $\breve{\mathcal{P}}$.

\begin{figure}[ht]
 \centering
\includegraphics{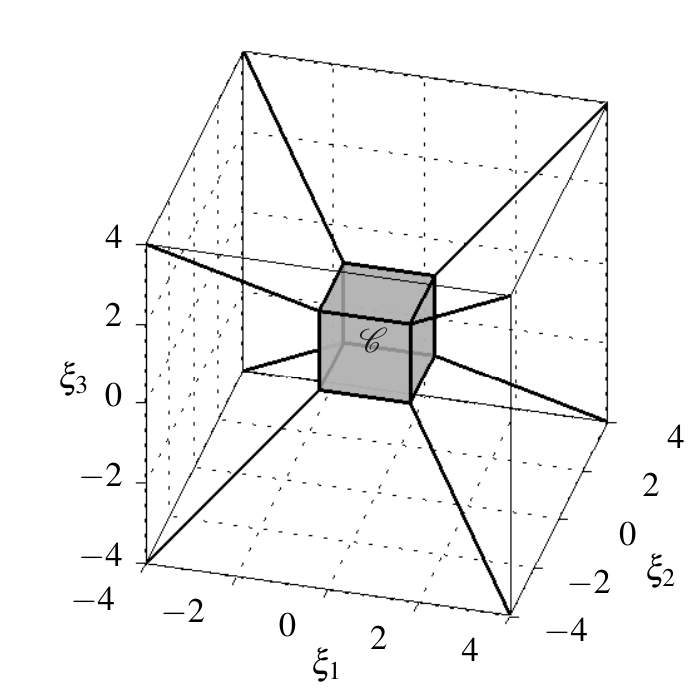}
      \caption{The partition of the frequency domain: The centered
cube $\cC$. The arrangement of the six pyramids is
          indicated by the `diagonal' lines. See
          Fig.~\ref{fig:pyramids} for a sketch of the pyramids.}
\label{fig:partition}
\end{figure}

The partitioning of frequency space into pyramids allows us to restrict the range of the shear parameters.
Without such a partitioning as, \eg in shearlet systems arising from the shearlet group, one must allow
arbitrarily large shear parameters, which leads to a treatment biased towards one axis. The defined partition
however enables restriction of the shear parameters to $[-\ceilsmall{2^{j/2}},\ceilsmall{2^{j/2}}]$,
similar to the definition of cone-adapted discrete shearlet systems. We would like to emphasize that this
approach is key to provide an almost uniform treatment of different directions in a sense of a `good'
approximation to rotation.

Pyramid-adapted discrete shearlets are scaled according to the \emph{paraboloidal scaling matrices} $A_{2^j}$,
$\tilde{A}_{2^j}$ or $\breve{A}_{2^j}$, $j \in \Z$ defined by
\[
  A_{2^j} = \begin{pmatrix} 2^j & 0 & 0 \\ 0 & \!\! 2^{j/2} & 0 \\ 0 & 0
    & \!\! 2^{j/2} \end{pmatrix},
  \quad
  \tilde{A}_{2^j} = \begin{pmatrix} 2^{j/2}\! & 0 & 0 \\ 0 & \!2^{j} & 0 \\ 0 & 0 & \!\!2^{j/2} \end{pmatrix},
  \; \text{and} \;\:\;
  \breve{A}_{2^j} = \begin{pmatrix} 2^{j/2}\! & 0 & 0 \\ 0 & \!\!2^{j/2} & 0 \\ 0 & 0 & \!\!2^{j} \end{pmatrix},
\]
and directionality is encoded by the \emph{shear matrices} $S_k$, $\tilde{S}_k$, or $\breve{S}_k$, $k = (k_1,k_2) \in \Z^2$,
given by
\[
S_k =\begin{pmatrix} 1\; & k_1\; & k_2 \\ 0 & 1 & 0\\ 0 & 0 & 1 \end{pmatrix},
\quad
\tilde{S}_k = \begin{pmatrix} 1 & 0 & 0 \\ k_1\; & 1\; & k_2 \\ 0 & 0 & 1 \end{pmatrix},
\quad \text{and} \quad
\breve{S}_k = \begin{pmatrix} 1 & 0 & 0 \\ 0 & 1 & 0 \\ k_1\; & k_2\; & 1 \end{pmatrix},
\]
respectively. The reader should note that these definitions are (discrete) special cases of the general setup
in \cite{coorbit}. The translation lattices will be defined through the following matrices: $M_c = \mathrm{diag}(c_1,c_2,c_2)$,
$\tilde{M}_c = \mathrm{diag}(c_2,c_1,c_2)$, and $\breve{M}_c = \mathrm{diag}(c_2,c_2,c_1)$, where $c_1>0$ and $c_2>0$.

We are now ready to introduce 3D shearlet systems, for which we will make use of the vector notation $\abs{k} \le K$
for $k = (k_1,k_2)$ and $K>0$ to denote $\abs{k_1} \le K$ \emph{and} $\abs{k_2} \le K$.

\begin{definition}
\label{def:discreteshearlets3d}
For $c=(c_1,c_2) \in (\R_+)^2$, the \emph{pyramid-adapted discrete shearlet system} $\SH(\phi,\psi,\tilde{\psi},\breve{\psi};c)$
generated by $\phi, \psi, \tilde{\psi}, \breve{\psi} \in L^2(\R^3)$ is defined by
\[
\SH(\phi,\psi,\tilde{\psi},\breve{\psi};c) = \Phi(\phi;c_1) \cup \Psi(\psi;c) \cup \tilde{\Psi}(\tilde{\psi};c) \cup \breve{\Psi}(\breve{\psi};c),
\]
where
\begin{align*}
\Phi(\phi;c_1) &= \setprop{\phi_m = \phi(\cdot-m)}{ m \in c_1\Z^3}, \\
\Psi(\psi;c) &= \setprop{\psi_{j,k,m} = 2^j {\psi}({S}_{k} {A}_{2^j}\cdot-m) }{ j \ge 0, |k| \le \ceilsmall{2^{j/2}}, m \in M_c \Z^3 }, \\
\tilde{\Psi}(\tilde{\psi};c) &= \{\tilde{\psi}_{j,k,m} = 2^j \tilde{\psi}(\tilde{S}_{k} \tilde{A}_{2^j}\cdot-m)
: j \ge 0, |k| \le \lceil 2^{j/2} \rceil, m \in \tilde{M}_c \Z^3 \},\\
\intertext{and}
\breve{\Psi}(\breve{\psi};c) &= \{\breve{\psi}_{j,k,m} = 2^j \breve{\psi}(\breve{S}_{k} \breve{A}_{2^j}\cdot-m)
: j \ge 0, |k| \le \lceil 2^{j/2} \rceil, m \in \breve{M}_c \Z^3 \},
\end{align*}
where $j \in \N_0$ and $k \in \Z^2$. For the sake of brevity, we will sometimes also use the notation $\psi_{\lambda}$
with $\lambda=(j,k,m)$.
\end{definition}

We now focus on two different special classes of pyramid-adapted discrete shearlets leading to the class of band-limited
shearlets and the class of compactly supported shearlets for which optimality of their approximation properties with
respect to cartoon-like images will be proven in Sect.~\ref{sec:optim-sparse-appr}.

\subsection{Band-Limited 3D Shearlets}
\label{subsec:bandlimited}

Let the shearlet generator $\psi \in L^2(\R^3)$ be defined by
\begin{equation}
 \label{eq:bandlimitedshearlets}
\hat \psi(\xi) = \hat \psi_1(\xi_1)\hat\psi_2\Bigl(\frac{\xi_2}{\xi_1}\Bigr)\hat \psi_2
\Bigl(\frac{\xi_3}{\xi_1}\Bigr),
\end{equation}
where $\psi_1$ and $\psi_2$ satisfy the following assumptions:
\begin{enumerate}[(i)]
\item $\hat \psi_1 \in C^{\infty}(\R)$, $\supp \hat \psi_1 \subset  \itvccs{-4}{-\frac{1}{2}} \cup \itvccs{\frac{1}{2}}{4}$, and
\begin{equation}\label{eq:calder}
\sum_{j \ge 0} \abs{\hat \psi_1 (2^{-j}\xi)}^2 = 1 \quad \text{for }
\abs{\xi} \ge 1, \xi \in \R. 
\end{equation}
\item $\hat \psi_2 \in C^{\infty}(\R)$, $\supp \hat \psi_2 \subset \itvcc{-1}{1}$, and
\begin{equation}\label{eq:shift}
\sum_{l=-1}^{1} \abs{\hat \psi_2(\xi+l)}^2 = 1 \quad
  \text{for } \abs{\xi} \leq 1, \xi \in \R.
\end{equation}
\end{enumerate}
Thus, in frequency domain, the band-limited function  $\psi \in L^2(\R^3)$ is almost a tensor product
of one wavelet with two `bump' functions, thereby a canonical generalization of the classical band-limited
2D shearlets, see also Chapter \cite{Intro}. This implies the support in frequency domain to have a needle-like
shape with the wavelet acting in radial direction ensuring high directional selectivity, see also Fig.~\ref{fig:tiling-3d}.
\begin{figure}[ht]
\centering
 \includegraphics{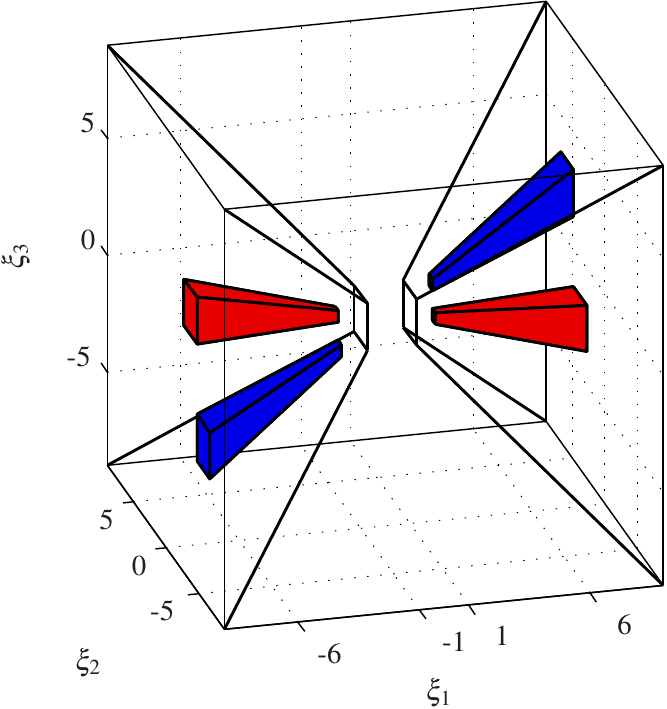}
\caption{Support of two shearlet elements $\psi_{j,k,m}$ in the
  frequency domain. The two shearlet elements have the same scale
  parameter $j=2$, but different shearing parameters $k=(k_1,k_2)$.}
  \label{fig:tiling-3d}
\end{figure}
The derivation from being a tensor product, i.e., the substitution of
$\xi_2$ and $\xi_3$ by the quotients $\xi_2/\xi_1$ and $\xi_3/\xi_1$,
respectively, in fact ensures a favorable behavior with respect to the
shearing operator, and thus a tiling of frequency domain which leads
to a tight frame for $L^2(\R^3)$.

A first step towards this result is the following observation.

\begin{theorem}[\cite{GL10_3d}]\label{thm:construction-bandlimited-3d}
Let $\psi$ be a band-limited shearlet defined as in this subsection. Then the family of functions
\[
\Psi(\psi) = \{\psi_{j,k,m} : j \ge 0, |k| \leq \ceilsmall{2^{j/2}}, m \in \tfrac{1}{8}\Z^3\}
\]
forms a tight frame for $\check L^2(\cP):=\{f \in L^2(\R^3): \supp \hat f \subset \cP\}$.
\end{theorem}

\begin{proof}
For each $j \ge 0$, equation \eqref{eq:shift} implies that
\[
\sum_{k = -\ceil{2^{j/2}}}^{\ceil{2^{j/2}}} |\hat \psi_2(2^{j/2}\xi + k)|^2 = 1, \quad \text{for}
\,\, |\xi| \leq 1.
\]
Hence, using equation \eqref{eq:calder}, we obtain
\begin{eqnarray*}
\lefteqn{\sum_{j \ge 0} \sum_{k_1,k_2=-\ceil{2^{j/2}}}^{\ceil{2^{j/2}}} |\hat \psi(S^T_kA^{-1}_{2^j}\xi)|^2}\\
& = & \sum_{j \ge 0} |\hat \psi_1 (2^{-j}\xi_1)|^2| \sum_{k_1=-\ceil{2^{j/2}}}^{\ceil{2^{j/2}}} |\hat \psi_2(2^{j/2}\tfrac{\xi_2}{\xi_1}+k_1)|^2
\sum_{k_2=-\ceil{2^{j/2}}}^{\ceil{2^{j/2}}} |\hat\psi_2(2^{j/2}\tfrac{\xi_2}{\xi_1}+k_2)|^2\\
& = & 1,
\end{eqnarray*}
for $\xi=(\xi_1,\xi_2,\xi_3) \in \cP$. Using this equation together with the fact that $\hat \psi$
is supported inside $\itvcc{-4}{4}^3$ proves the theorem. 
\end{proof}

By Thm.~\ref{thm:construction-bandlimited-3d} and a change of variables,
we can construct shearlet frames for $\check L^2(\cP)$, $\check
L^2(\tilde\cP)$, and $\check L^2(\breve\cP)$, respectively.
Furthermore, wavelet theory provides us with many choices of $\phi \in
L^2(\R^3)$ such that $\Phi(\phi;\tfrac{1}{8})$ forms a frame for
$\check L^2(\cC)$. Since $\R^3 = \cC \cup \cP \cup \tilde \cP \cup
\breve \cP$ as a disjoint union, we can express any function $f \in
L^2(\R^3)$ as $f = P_\cC f + P_{\cP} f + P_{\tilde \cP} f + P_{\breve
  \cP} f$, where each component corresponds to the orthogonal
projection of $f$ onto one of the three pairs of pyramids or the
centered cube in the frequency space. We then expand each of these
components in terms of the corresponding tight frame. Finally, our
representation of $f$ will then be the sum of these four expansions.
We remark that the projection of $f$ onto the four subspaces can lead
to artificially slow decaying shearlet coefficients; this will, \eg be
the case if $f$ is in the Schwartz class. This problem does in fact
not occur in the construction of compactly supported shearlets.

\subsection{Compactly Supported 3D Shearlets}
\label{subsec:compactsupport}

It is easy to see that the general form
\eqref{eq:bandlimitedshearlets} does never lead to a function which is
compactly supported in spatial domain. Thus, we need to deviate this
form by now taking indeed exact tensor products as our shearlet
generators, which has the additional benefit of leading to fast
algorithmic realizations. This however causes the problem that the
shearlets do not behave as favorable with respect to the shearing
operator as in the previous subsection, and the question arises
whether they actually do lead to at least a frame for $L^2(\R^3)$. The
next results shows this to be true for an even much more general form
of shearlet generators including compactly supported separable
generators. The attentive reader will notice that this theorem even covers
the class of band-limited shearlets introduced in
Sect.~\ref{subsec:bandlimited}.

\begin{theorem}[\!\cite{KLL10}]\label{thm:construction_suff}
Let $\phi, \psi \in L^2({\mathbb R}^3)$ be functions such that
\[
|\hat\phi(\xi)| \le C_1 \min\{1,|\xi_1|^{-\gamma}\} \cdot \min\{1,|\xi_2|^{-\gamma}\} \cdot \min\{1,|\xi_3|^{-\gamma}\},
\]
and
\begin{multline*}
|\hat\psi(\xi)| \le C_2 \cdot \min\{1,|\xi_1|^{\delta}\} \cdot \min\{1,|\xi_1|^{-\gamma}\} \cdot \min\{1,|\xi_2|^{-\gamma}\} \cdot \min\{1,|\xi_3|^{-\gamma}\},
\end{multline*}
for some constants $C_1,C_2 > 0$ and $\delta>2\gamma>6$. Define $\tilde{\psi}(x) = \psi(x_2,x_1,x_3)$ and $\breve{\psi}(x) =
\psi(x_3,x_2,x_1)$ for $x=(x_1,x_2,x_3) \in \R^3$. Then there exists a
constant $c_0>0$ such that the shearlet system
$SH(\phi,\psi,\tilde{\psi},\breve{\psi};c)$ forms a frame for $L^2({\mathbb R}^3)$ for all $c=(c_1,c_2)$ with $c_2 \leq c_1 \leq
c_0$ provided that there exists a positive constant $M>0$ such that
\begin{align}
\label{eq:lower-Linf-bound}
|\hat\phi(\xi)|^2 + \sum_{j \ge 0}\sum_{k_1,k_2 \in K_j} |\hat \psi({S_k^T} A_{2^{j}}\xi)|^2+|\hat{\tilde{\psi}}(\tilde S_k^T
\tilde{A}_{2^{j}}\xi)|^2 +|\hat{\breve{\psi}}(\breve S_k^T \breve{A}_{2^{j}}\xi)|^2 > M
\end{align}
for a.e $\xi \in {\mathbb R}^3$, where $K_j:=\itvcc{-\ceilsmall{2^{j/2}}}{\ceilsmall{2^{j/2}}}$.
\end{theorem}

We next provide an example of a family of compactly supported shearlets satisfying the assumptions of Thm.~\ref{thm:construction_suff}.
However, for applications, one is typically not only interested in whether a system forms a frame, but in the ratio of the associated frame
bounds. In this regard, these shearlets also admit a theoretically derived estimate for this ratio which is reasonably close to $1$, i.e.,
to being tight. The numerically derived ratio is even significantly closer as expected.

\begin{example}
\label{example:compact-for-pyramid-3D}
Let $K, L \in \N$ be such that $L \ge 10$ and $\frac{3L}{2} \le K \le 3L-2$, and define a shearlet $\psi \in L^2(\R^3)$ by
\begin{equation}
\hat{\psi}(\xi) =
m_1(4\xi_1)\hat{\phi}(\xi_1)\hat{\phi}(2\xi_2)\hat\phi(2\xi_3), \quad
\xi = (\xi_1,\xi_2,\xi_3) \in \R^3,\label{eq:def-psi-compact}
\end{equation}
where the function $m_0$ is the low pass filter satisfying
\[
|m_0(\xi_1)|^2 = \cos^{2K}(\pi\xi_1)) \sum_{n=0}^{L-1} \binom{K-1+n}{ 
  n} \sin^{2n}(\pi\xi_1),
\]
for $\xi_1 \in \R,$ the function $m_1$ is the associated bandpass filter defined by
\[
|m_1(\xi_1)|^2 = |m_0(\xi_1+1/2)|^2, \quad \xi_1 \in \R,
\]
and $\phi$  the scaling function is given by
\[
\hat{\phi}(\xi_1) = \prod_{j=0}^{\infty} m_0(2^{-j}\xi_1), \quad \xi_1 \in \R.
\]

In \cite{KKL10a,KLL10} it is shown that $\phi$ and $\psi$ indeed are compactly
supported. Moreover, we have the following result.
\begin{theorem}[\cite{KLL10}]
  Suppose $\psi \in L^2(\R^3)$ is defined as in
  (\ref{eq:def-psi-compact}). Then there exists a sampling constant
  $c_0>0$ such that the shearlet system $\Psi(\psi;c)$ forms a frame
  for $\check L^2(\cP)$ for any translation matrix $M_c$ with
  $c=(c_1,c_2) \in (\R_+)^2$ and $c_2 \le c_1 \le {c}_0$.
\end{theorem}
\begin{proof}[sketch]
  Using upper and lower estimates of the absolute value of the
  trigonometric polynomial $m_0$ (cf.~\cite{Dau92,KKL10a}), one can
  show that $\psi$ satisfies the hypothesis of
  Thm.~\ref{thm:construction_suff} as well as
  \[ \sum_{j \ge 0}\sum_{k_1,k_2 \in K_j} |\hat \psi({S_k^T}
  A_{2^{j}}\xi)|^2 > M \qquad \text{for all $\xi \in \cP$,}\] where
  $M>0$ is a constant, for some sufficiently small $c_0>0$. We note
  that this
  inequality is an analog to~(\ref{eq:lower-Linf-bound}) for the
  pyramid $\cP$. Hence, by a result similar to
  Thm.~\ref{thm:construction_suff}, but for the case, where we
  restrict to the pyramid $\check L^2(\cP)$, it then follows that
  $\Psi(\psi;c)$ is a frame.
\end{proof}

To obtain a frame for all of $L^2(\R^3)$ we simply set
$\tilde{\psi}(x) = \psi(x_2,x_1,x_3)$ and $\breve{\psi}(x) =
\psi(x_3,x_2,x_1)$ as in Thm.~\ref{thm:construction_suff}, and choose
$\phi(x)=\phi(x_1)\phi(x_2)\phi(x_3)$ as scaling function for
$x=(x_1,x_2,x_3) \in \R^3$. Then the corresponding shearlet system
$SH(\phi,\psi,\tilde{\psi},\breve{\psi};c,\scp)$ forms a frame for
$L^2(\R^3)$. The proof basically follows from Daubechies' classical
estimates for wavelet frames in \cite[\S 3.3.2]{Dau92} and the fact
that anisotropic and sheared windows obtained by applying the scaling
matrix $A_{2^j}$ and the shear matrix $S^T_k$ to the effective
support\footnote{Loosely speaking, we say that $f \in L^2(\R^d)$
    has \emph{effective} support on $B$ if the ratio
    $\norm[L^2]{f\chi_B}/\norm[L^2]{f}$ is ``close'' to $1$.} of
$\hat{\psi}$ cover the pyramid $\cP$ in the frequency domain. The same
arguments can be applied to each of shearlet generators $\psi$,
$\tilde \psi$ and $\breve{\psi}$ as well as the scaling function
$\phi$ to show a covering of the entire frequency domain and thereby
the frame property of the pyramid-adapted shearlet system for
$L^2(\R^3)$. We refer to \cite{KLL10} for the detailed proof.

Theoretical and numerical estimates of frame bounds for a particular parameter choice are shown in
Table~\ref{tab:frame-bounds-3d}. We see that the theoretical
  estimates are overly pessimistic, since they are a factor $20$ larger
  than the numerical estimated frame bound ratios. We mention that
for 2D the estimated frame bound ratios are approximately $1/10$ of the
ratios found in Table~\ref{tab:frame-bounds-3d}.



\begin{table}
\caption{Frame bound ratio for the shearlet frame from
  Example~\ref{example:compact-for-pyramid-3D} with parameters $K=39, L=19$.}
\label{tab:frame-bounds-3d}       
%
%
\begin{tabular}{p{3.5cm}p{3cm}p{5.cm}}
\hline\noalign{\smallskip}
Theoretical ($B/A$)  &  Numerical ($B/A$)  & Translation
constants ($c_1,c_2$) \\
\noalign{\smallskip}\hline\noalign{\smallskip}
345.7                     &            13.42               &                        (0.9,   0.25)\\

226.6                 &                13.17        &                               (0.9,   0.20)\\
226.4                     &            13.16     &                                  (0.9,   0.15)\\
226.4                     &            13.16         &                              (0.9,   0.10)\\
\noalign{\smallskip}\hline\noalign{\smallskip}
\end{tabular}
\end{table}

\end{example}

\subsection{Some Remarks on Construction Issues}


The compactly supported shearlets $\psi_{j,k,m}$ from
  Example~\ref{example:compact-for-pyramid-3D} are, in spatial domain,
  of size $2^{-j/2}$ times $2^{-j/2}$ times $2^{-j}$ due to the
  scaling matrix $A_{2^j}$. This reveals that the shearlet elements
  will become `plate-like' as $j \to \infty$. For an illustration, we
  refer to Fig.~\ref{fig:support-size-3d-shearlet}. Band-limited
  shearlets, on the other hand, do not have compactly support, but
  their effective support (the region where the energy of the function
  is concentrated) in spatial domain will likewise be of size $2^{-j/2}$
  times $2^{-j/2}$ times $2^{-j}$ owing to their smoothness in frequency domain.
\begin{figure}[ht]
\centering
\includegraphics[width=7cm]{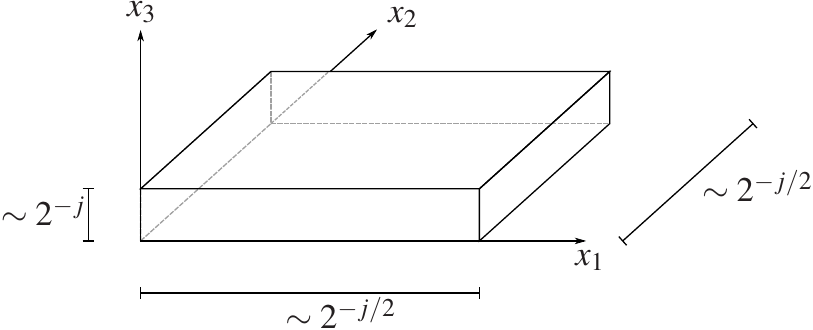}
\caption{Support of a shearlet $\breve{\psi}_{j,0,m}$ from
  Example~\ref{example:compact-for-pyramid-3D}.}
\label{fig:support-size-3d-shearlet}
\end{figure}
Contemplating about the fact that intuitively such shearlet elements
should provide sparse approximations of surface singularities, one
could also think of using the scaling matrix
$A_{2^j}=\diag{2^j,2^j,2^{j/2}}$ with similar changes for $\tilde
A_{2^j}$ and $ \breve A_{2^j}$ to derive `needle-like' shearlet
elements in space domain. These would intuitively behave favorable with respect to the
other type of anisotropic features occurring in 3D, that is
curvilinear singularities. Surprisingly, we will show in Sect.~\ref{sec:extensions-3d} that for optimally sparse approximation plate-like
shearlets, i.e., shearlets associated with scaling matrix
$A_{2^j}=\diag{2^j,2^{j/2},2^{j/2}}$, and similarly $\tilde A_{2^j}$
and $ \breve A_{2^j}$ are sufficient.

Let us also mention that, more generally, non-paraboloidal scaling
matrices of the form $A_{j}=\diag{2^j,2^{a_1 j},2^{a_2 j}}$ for $0 <
a_1, a_2 \le 1$ can be considered. The parameters $a_1$ and $a_2$
allow precise control of the aspect ratio of the shearlet elements,
ranging from very plate-like to very needle-like, according to the
application at hand, \ie choosing the shearlet-shape that is the best
matches the geometric characteristics of the considered data. The case
$a_i <1$ is covered by the setup of the multidimensional shearlet
transform explained in Chapter \cite{coorbit}.

Let us finish this section with a general thought on the construction
of band-limited (not separable) tight shearlet frames versus compactly
supported (non-tight, but separable) shearlet frames. It seems that
there is a trade-off between {\em compact support} of the shearlet
generators, {\em tightness} of the associated frame, and {\em
  separability} of the shearlet generators. In fact, even in 2D, all
known constructions of tight shearlet frames do not use separable
generators, and these constructions can be shown to {\em not} be
applicable to compactly supported generators. Presumably, tightness is
difficult to obtain while allowing for compactly supported generators,
but we can gain separability which leads to fast algorithmic
realizations, see Chapter \cite{shearlab}. If we though allow
non-compactly supported generators, tightness is possible as shown in
Sect.~\ref{subsec:bandlimited}, but separability seems to be out of
reach, which causes problems for fast algorithmic realizations.

\section{Optimal Sparse Approximations}
\label{sec:optim-sparse-appr}

In this section, we will show that shearlets -- both band-limited as
well as compactly supported as defined in
Sect.~\ref{sec:pyram-adapt-shearl} -- indeed provide the optimal
sparse approximation rate for cartoon-like images from
Sect.~\ref{sec:optimal-sparsity}. Thus, letting
$(\psi_{\lambda})_\lambda=(\psi_{j,k,m})_{j,k,m}$ denote the
band-limited shearlet frame from Sect.~\ref{subsec:bandlimited} and the
compactly supported shearlet frame from Sect.~\ref{subsec:compactsupport} in
both 2D and 3D (see \cite{Intro}) and $d \in \{2, 3\}$, we aim to
prove that
\[
 \norm[L^2]{f-f_N}^2 \lesssim N^{-\frac{2}{d-1}} \quad \text{for all } f \in \cE^{2}_{L}(\R^d),
\]
where -- as debated in Sect.~\ref{sec:non-line-appr} -- $f_N$ denotes the
$N$-term approximation using the $N$ largest coefficients as
in~(\ref{eq:frame-n-term-largest}). Hence, in 2D we aim for the rate
$N^{-2}$ and in 3D we aim for the rate $N^{-1}$ with ignoring
$\log$-factors. As mentioned in Sect.~\ref{sec:optimal-sparsity},
see~(\ref{eq:sought-sparsity}), in order to prove these rate, it
suffices to show that the $n$th largest shearlet coefficient
$c^*_{\,n}$ decays as
\[
 \abs{c^*_{\, n}} \lesssim n^{-\frac{d+1}{2(d-1)}} = \begin{cases}
    n^{-3/2} & : \quad d=2, \\
   n^{-1} & : \quad d=3.
  \end{cases}
\]

According to Dfn.~\ref{def:optimal} this will show that among all
adaptive and non-adaptive representation systems shearlet frames
behave optimal with respect to sparse approximation of cartoon-like
images. That one is able to obtain such an optimal approximation error
rate might seem surprising, since the shearlet system as well as the
approximation procedure will be non-adaptive.

To present the necessary hypotheses, illustrate the key ideas of the
proofs, and debate the differences between the arguments for
band-limited and compactly supported shearlets, we first focus on the
situation of 2D shearlets. We then discuss the 3D situation, with a
sparsified proof, mainly discussing the essential differences to the
proof for 2D shearlets and highlighting the crucial nature of this
case (cf.~Sect.~\ref{subsec:3D}).

\subsection{Optimal Sparse Approximations in 2D}
\label{sec:optim-sparse-appr-2d}

As discussed in the previous section, in the case $d=2$, we aim for the
estimates $ \abs{c^*_{\, n}} \lesssim n^{-3/2}$ and
$\norm[L^2]{f-f_N}^2 \lesssim N^{-2}$ (up to log-factors). In
Sect.~\ref{sec:heuristic-analysis} we will first provide a heuristic
analysis to argue that shearlet frames indeed can deliver these rates. In
Sect.~\ref{sec:required-hypotheses} and \ref{sec:main-result} we then
discuss the required hypotheses and state the main optimality result.
The subsequent subsections are then devoted to proving the main result.

\subsubsection{A Heuristic Analysis}
\label{sec:heuristic-analysis}

We start by giving a heuristic argument (inspired by a similar
argument for curvelets in \cite{CD04}) on why the error $\norm[L^2]{f-f_N}^2$
satisfies the asymptotic rate $N^{-2}$. We emphasize that this heuristic argument
applies to both the band-limited and also the compactly supported
case.

For simplicity we assume $L=1$, and let $f \in \cE_L^2(\R^2)$ be a 2D
cartoon-like image. The main concern is to derive the
estimate~(\ref{eq:2d-sparsity-rate}) for the shearlet coefficients
$\innerprod{f}{\mathring{\psi}_{j,k,m}}$, where $\mathring{\psi}$
denotes either $\psi$ or $\tilde\psi$. We consider only the case
$\mathring{\psi} = \psi$, since the other case can be handled similarly. For compactly supported
shearlet, we can think of our generators having the form
$\psi(x)=\eta(x_1)\phi(x_2)$, $x=(x_1,x_2)$, where $\eta$ is a wavelet
and $\phi$ a bump (or a scaling) function. It will become important,
that the wavelet `points' in the $x_1$-axis direction, which
corresponds to the `short' direction of the shearlet. For band-limited
generators, we can think of our generators having the form
$\hat\psi(\xi)=\hat\eta(\xi_2/\xi_1)\hat\phi(\xi_2)$ for
$\xi=(\xi_1,\xi_2)$. 
We, moreover, restrict our analysis to shearlets $\psi_{j,k,m}$ since
the frame elements $\tilde{\psi}_{j,k,m}$ can be handled in a similar
way.

   We now consider three cases of coefficients
  $\innerprod{f}{\psi_{j,k,m}}$:
  \begin{enumerate}[(a)]
  \item Shearlets $\psi_{j,k,m}$ whose support does not overlap
    with the boundary $\partial B$.
  \item Shearlets $\psi_{j,k,m}$ whose support overlaps with $\partial B$
    and is nearly tangent.
  \item Shearlets $\psi_{j,k,m}$ whose support overlaps with
    $\partial B$, but not tangentially.
  \end{enumerate}

 \begin{figure}[ht]
\centering
\includegraphics{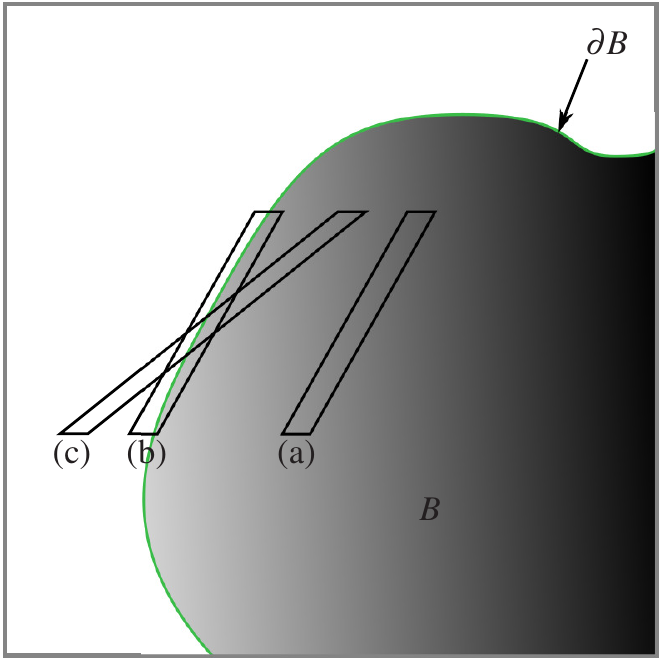}
\caption{Sketch of the three cases: (a) the support of $\psi_{j,k,m}$
  does not overlap with $\partial B$, (b) the support of
  $\psi_{j,k,m}$ does overlap with $\partial B$ and is nearly tangent,
  (c) the support of $\psi_{j,k,m}$ does overlap with $\partial B$,
  but not tangentially. Note that only a section of the discontinuity
  curve $\partial B$ is shown, and that for the case of band-limited
  shearlets only the effective support is shown.}
\label{fig:case-a-b-c}
\end{figure}

It turns out that only coefficients from case~(b) will be
significant. Case~(b) is, loosely speaking, the situation, where
the wavelet $\eta$ crosses the discontinuity curve over the
entire `height' of the shearlet, see Fig.~\ref{fig:case-a-b-c}. 

\emph{Case (a).} Since $f$ is $C^2$-smooth away from $\partial B$, the
coefficients $\abs{\innerprod{f}{\psi_{j,k,m}}}$ will be sufficiently
small owing to the approximation property of the wavelet $\eta$.  The situation
is sketched in Fig.~\ref{fig:case-a-b-c}.

\emph{Case (b).} At scale $j>0$, there are about $O(2^{j/2})$
coefficients, since the shearlet elements are of length $2^{-j/2}$
(and `thickness' $2^{-j}$) and the length of $\pa B$ is finite.
By H\"older's inequality, we immediately obtain
\[ \abs{\innerprod{f}{\psi_{j,k,m}}} \le \norm[L^\infty]{f}
  \norm[L^1]{\psi_{j,k,m}} \le C_1 \, 2^{-3j/4}
  \norm[L^1]{\psi} \le C_2 \cdot 2^{-3j/4} \]
for some constants $C_1, C_2>0$. In other words, we have $O(2^{j/2})$
coefficients bounded by $C_2 \cdot 2^{-3j/4}$. Assuming the case (a) and
(c) coefficients are negligible, the $n$th largest coefficient
$c^\ast_{\,n}$ is then bounded by
\[ \abs{c^\ast_{\,n}} \le C \cdot n^{-3/2},\] which was what we aimed
to show; compare to (\ref{eq:coeff-fastest-decay-upper}) in Dfn.
\ref{def:optimal}. This in turn implies (cf.~estimate~(\ref{eq:sparsity-implies-approx})) that
\[
    \sum_{n>N} \abs{c^\ast_{\,n}}^2 \leq \sum_{n>N}C \cdot n^{-3} \le
    C \cdot \int_N^\infty x^{-3} \D x \le C \cdot N^{-2}.
 \]
By Lemma~\ref{lemma:n-term-frame-approx}, as desired it follows that
  \[
  \norm[L^2]{f-f_N}^2 \leq \frac{1}{A}\sum_{n>N} \abs{c^\ast_{\,n}}^2 \le C \cdot N^{-2},
  \]
  where $A$ denotes the lower frame bound of the shearlet frame.

  \emph{Case (c).} Finally, when the shearlets are sheared away from
  the tangent position in case~(b), they will again be small. This is
  due to the frequency support of $f$ and $\psi_\lambda$ as well as to
  the directional vanishing moment conditions assumed in
  Setup~\ref{setup:assumptions-on-generators-BL-2D}
  or~\ref{setup:assumptions-on-generators-CS-2D}, which will be formally introduced in the next
  subsection.

  Summarising our findings, we have argued, at least heuristically, that shearlet
  frames provide optimal sparse approximation of cartoon-like images
  as defined in Dfn.~\ref{def:optimal}.


\subsubsection{Required Hypotheses}
\label{sec:required-hypotheses}

After having build up some intuition on why the optimal sparse
approximation rate is achievable using shearlets, we will now go into
more details and discuss the hypotheses required for the main result.
This will along the way already highlight some differences between the
band-limited and compactly supported case.


\begin{figure}[ht]
\centering
 \includegraphics[width=0.95\textwidth]{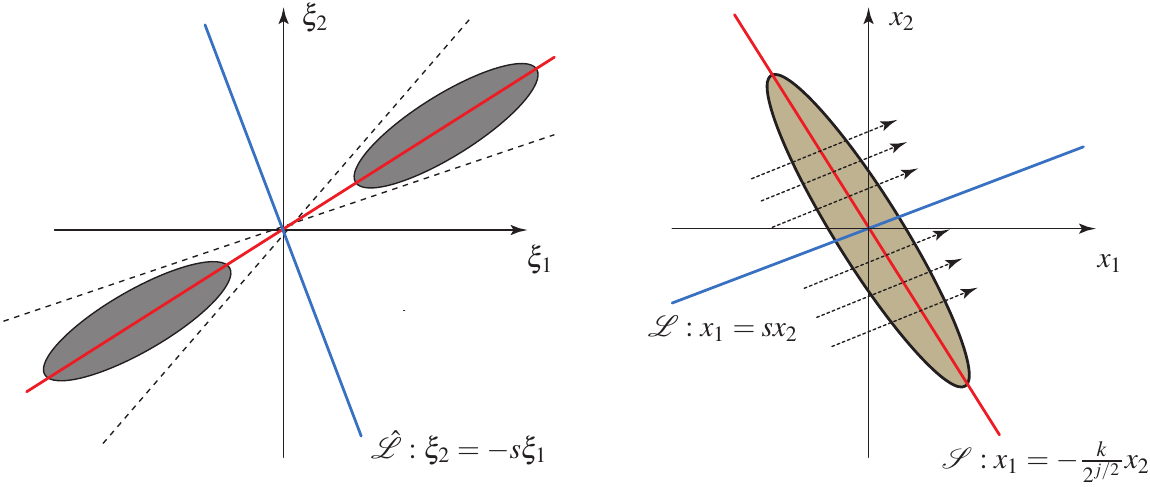}
\begin{minipage}[t]{0.4\linewidth}
   \caption{Shaded region: The effective part of $\supp \hat
      \psi_{j,k,m}$ in the frequency domain.}
    \label{fig:directional-van-mom-freq}
\end{minipage}
\hspace{0.1\textwidth} 
\begin{minipage}[t]{0.4\linewidth}
  \caption{Shaded region: The effective part of $\supp \psi_{j,k,m}$
    in the spatial domain. Dashed lines: the direction of line
    integration $I(t)$.}
  \label{fig:directional-van-mom-spatial}
\end{minipage}
\end{figure}

For this discussion, assume that $f \in L^2(\R^2)$ is piecewise
$C^{L+1}$-smooth with a discontinuity on the line $\cL : x_1 = sx_2$,
$s\in \R$, so that the function $f$ is well approximated by two 2D
  polynomials of degree $L>0$, one polynomial on either side of $\cL$, and denote
  this piecewise polynomial $q(x_1,x_2)$. We denote the restriction of
  $q$ to lines $x_1=sx_2+t$, $t\in \R$, by $p_t(x_2)=q(sx_2+t,x_2)$.
  Hence, $p_t$ is a 1D polynomial along lines parallel to $\cL$ going
  through $(x_1,x_2)=(t,0)$; these lines are marked by dashed lines in
  Fig.~\ref{fig:directional-van-mom-spatial}.

We now aim at estimating the absolute value of a
shearlet coefficient $\innerprod{f}{\psi_{j,k,m}}$ by
\begin{equation} \label{eq:boundingcoeffs} \abs{\innerprod{f}{\psi_{j,k,m}}} \leq
\abs{\innerprod{q}{\psi_{j,k,m}}} + \abs{\innerprod{(q-f)}{\psi_{j,k,m}}}.
\end{equation} We first observe that $\abs{\innerprod{f}{\psi_{j,k,m}}}$ will be
small depending on the approximation quality of the (piecewise) polynomial $q$ and
the decay of $\psi$ in the spatial domain. Hence it suffices to focus
on estimating $\abs{\innerprod{q}{\psi_{j,k,m}}}$.

For this, let us consider the line integration along the direction
$(x_1,x_2)=(s,1)$ as follows: For $t \in \R$ fixed, define integration
of $q \psi_{j,k,m}$ along the lines $x_1=sx_2+t$, $x_2 \in \R$, as
\[
I(t) = \int_{\R} p_t(x_2)\psi_{j,k,m}(sx_2+t,x_2)\D x_2,
\]
Observe that $\abs{\innerprod{q}{\psi_{j,k,m}}} = 0$ is equivalent to $I \equiv 0$. For simplicity, let us now assume $m = (0,0)$. Then
\begin{align*}
  I(t) &= 2^{\frac{3}{4}j}\int_{\R} p_t(x_2) \psi(S_k A_{2^j}(sx_2+t,x_2))\D x_2 \\
  &= 2^{\frac{3}{4}j} \sum_{\ell =0}^{L} c_{\ell} \int_{\R}x_2^{\ell} \psi(S_kA_{2^j}(sx_2+t,x_2))\D x_2\\
  &= 2^{\frac{3}{4}j}\sum_{\ell
    =0}^{L}c_{\ell}\int_{\R}x_2^{\ell}\psi(A_{2^j}S_{k/2^{j/2}+s}(t,x_2))\D x_2,
\end{align*}
and, by the Fourier slice theorem~\cite{Kak88} (see also \eqref{eq:fourier-slice-thm}), it follows that
\[
|I(t)| = 2^{\frac{3}{4}j}\Bigl|\sum_{\ell=0}^{L}\frac{2^{-\frac{\ell}{2}j}}{(2\pi)^{\ell}}c_{\ell}\int_{\R}\Bigl(\frac{\partial}{\partial \xi_2}\Bigr)^{\ell}
\hat \psi(A^{-1}_{2^j}S^{-T}_{k/2^{j/2}+s}(\xi_1,0))\E^{2\pi i \xi_1 t}\D\xi_1\Bigr|.
\]
Note that
\[
\int_{\R}\Bigl(\frac{\partial}{\partial \xi_2}\Bigr)^{\ell}
\hat \psi(A^{-1}_{2^j}S^{-T}_{k/2^{j/2}+s}(\xi_1,0))\E^{2\pi i \xi_1 t}\D\xi_1 = 0 \quad \text{for almost all}\,\, t \in \R
\]
if and only if
\[
\Bigl(\frac{\partial}{\partial \xi_2}\Bigr)^{\ell}
\hat \psi(A^{-1}_{2^j}S^{-T}_{k/2^{j/2}+s}(\xi_1,0)) = 0 \quad \text{for almost all} \,\, \xi_1 \in \R.
\]
Therefore, to ensure $I(t) = 0$ for any 1D polynomial $p_t$ of degree $L>0$, we require the following condition:
\[
\Bigl(\frac{\partial}{\partial \xi_2}\Bigr)^{\ell}
\hat \psi_{j,k,0}(\xi_1,-s\xi_1) = 0 \quad \text{for almost all }
\xi_1 \in \R  \text{ and }  \ell = 0,\dots, L.
\]
These are the so-called \emph{directional vanishing moments} (cf.~\cite{DV05}) in the direction $(s,1)$.
We now consider the two cases, band-limited shearlets and compactly supported shearlets, separately.

If $\psi$ is a band-limited shearlet generator, we automatically have
\begin{equation}\label{eq:band}
\Bigl( \frac{\partial}{\partial \xi_2}\Bigr)^{\ell} \hat \psi_{j,k,m}(\xi_1,-s\xi_1) = 0 \quad \text{for}\,\,
\ell = 0,\dots,L \quad \text{if}\,\, |s+\frac{k}{2^{j/2}}|\ge 2^{-j/2},
\end{equation}
since $\supp \hat \psi \subset \cD$, where $\cD = \{\xi \in \R^2 :
|\xi_2/\xi_1| \leq 1\}$ as discussed in Chapter~\cite{Intro}. Observe
that the `direction' of $\supp \psi_{j,k,m}$ is determined by the line
$\cS: x_1 = -\frac{k}{2^{j/2}}x_2$. Hence, equation \eqref{eq:band}
implies that, if the direction of $\supp \psi_{j,k,m}$, i.e., of
$\cS$ is {\em not} close to the direction of $\cL$ in the sense that
$|s+\frac{k}{2^{j/2}}|\ge 2^{-j/2}$, then
\[
|\innerprod{q}{\psi_{j,k,m}}| = 0.
\]

However, if $\psi$ is a compactly supported shearlet generator,
equation \eqref{eq:band} can never hold, since it requires that $\supp
\hat \psi \subset \cD$. Therefore, for compactly supported generators,
we will assume that $(\frac{\partial}{\partial \xi_2})^l\hat\psi$,
$l=0,1$, has sufficient decay in $\cD^{c}$ to force $I(t)$ and hence
$\abssmall{\innerprod{q}{\psi_{j,k,m}}}$ to be sufficiently small. It should be
emphasized that the drawback that $I(t)$ will only be `small' for
compactly supported shearlets (due to the lack of exact directional
vanishing moments) will be compensated by the perfect localization
property which still enables optimal sparsity.


Thus, the developed conditions ensure that both terms on the right hand side of \eqref{eq:boundingcoeffs}
can be effectively bounded.

\bigskip

This discussion gives naturally rise to the following hypotheses for
optimal sparse approximation. Let us start with the hypotheses for the
band-limited case.

\begin{setup}\label{setup:assumptions-on-generators-BL-2D}
  The generators $\phi, \psi, \tilde{\psi} \in L^2(\R^2)$ are
  band-limited and $C^\infty$ in the frequency domain.
  Furthermore, the shearlet system $SH(\phi,\psi,\tilde{\psi};c)$ forms
  a frame for $L^2(\R^2)$
  (cf. the construction in Chapter~\cite{Intro} or
  Sect.~\ref{subsec:bandlimited}).
\end{setup}

In contrast to this, the conditions for the compactly supported
shearlets are as follows:

\begin{setup}\label{setup:assumptions-on-generators-CS-2D}
  The generators $\phi, \psi, \tilde{\psi} \in
  L^2(\R^2)$ are compactly supported, and the shearlet system
  $SH(\phi,\psi,\tilde{\psi};c)$ forms a frame for
  $L^2(\R^2)$. Furthermore, for all $\xi =
  (\xi_1,\xi_2) \in \R^2$, the function $\psi$ satisfies
      \begin{enumerate}[(i)]
      \item $|\hat\psi(\xi)| \le C \cdot \min\{1,|\xi_1|^{\delta}\}
        \cdot
        \min\{1,|\xi_1|^{-\gamma}\} \cdot
        \min\{1,|\xi_2|^{-\gamma}\}$, and
      \item $\left|\frac{\partial}{\partial \xi_2}\hat
          \psi(\xi)\right| \le |h(\xi_1)|
        \left(1+\frac{|\xi_2|}{|\xi_1|}\right)^{-\gamma},$
      \end{enumerate}
      where $\delta > 6$, $\gamma \ge 3$, $h \in
      L^1(\R)$, and $C$ a constant, and $\tilde{\psi}$
      satisfies analogous conditions with the obvious
      change of coordinates (cf. the construction in
Sect.~\ref{subsec:compactsupport}).
  \end{setup}

Conditions~(i) and (ii) in
Setup~\ref{setup:assumptions-on-generators-CS-2D} are exactly the
decay assumptions on $(\frac{\partial}{\partial
  \xi_2})^l \hat\psi$, $l=0,1$, discussed above that guarantees control
of the size of $I(t)$.

\subsubsection{Main Result}
\label{sec:main-result}

We are now ready to present the main result, which states that under
Setup~\ref{setup:assumptions-on-generators-BL-2D}
or Setup~\ref{setup:assumptions-on-generators-CS-2D} shearlets provide
optimally sparse approximations for cartoon-like images.

\begin{theorem}[\cite{GL07,KL10}]
\label{thm:opt-sparse-2D}
Assume Setup~\ref{setup:assumptions-on-generators-BL-2D}
or~\ref{setup:assumptions-on-generators-CS-2D}. Let $L\in \N$. For any
$\cp > 0$ and $\mu >0$, the shearlet frame
$SH(\phi,\psi,\tilde{\psi};c)$ provides optimally sparse
approximations of functions $f \in \cE_L^2(\R^2)$ in the sense of Dfn. \ref{def:optimal}, i.e.,
\begin{align}
  \norm[L^2]{f-f_N}^2 &= O(N^{-2}(\log{N})^3), \qquad &\text{as $N \to \infty$,}
\label{eq:2d-approx-rate}
\intertext{and}
   \abs{c^\ast_{\,n}} &\lesssim n^{-3/2} (\log
   n)^{3/2}, \qquad &\text{as $n \to \infty$,}\label{eq:2d-sparsity-rate}
\end{align}
 where $c=\setprop{\innerprod{f}{\mathring{\psi}_\lambda}}{\lambda \in
     \varLambda, \mathring{\psi}=\psi \text{ or } \mathring{\psi}=
     \tilde\psi}$ and $c^\ast=(c^\ast_n)_{n\in \N}$ is a decreasing
   (in modulus) rearrangement of $c$.
\end{theorem}

\subsubsection{Band-Limitedness versus Compactly Supportedness}
\label{sec:band-limited-versus-comp}

Before we delve into the proof of Thm.~\ref{thm:opt-sparse-2D}, we
first carefully discuss the main differences between band-limited
shearlets and compactly supported shearlets which requires adaptions
of the proof.

In the case of compactly supported shearlets, we can consider the two
cases $|\supp \mathring\psi_{\lambda} \cap \partial B| \neq 0$ and
$|\supp \mathring\psi_{\lambda} \cap \partial B| = 0$. In case the
support of the shearlet intersects the discontinuity curve $\partial
B$ of the cartoon-like image $f$, we will estimate each shearlet
coefficient $\innerprod{f}{\mathring{\psi}_{\lambda}}$ individually
using the decay assumptions on $\hat \psi$ in
Setup~\ref{setup:assumptions-on-generators-CS-2D}, and then apply a
simple counting estimate to obtain the sought
estimates~(\ref{eq:2d-approx-rate}) and~(\ref{eq:2d-sparsity-rate}).
In the other case, in which the shearlet does not interact with the
discontinuity, we are simply estimating the decay of shearlet
coefficients of a $C^2$ function. The argument here is similar to the
approximation of smooth functions using wavelet frames and rely on
estimating coefficients at all scales using the frame property.

In the case of band-limited shearlets, it is not allowed to consider
two cases $|\supp \psi_{\lambda} \cap \partial B| = 0$ and $|\supp
\psi_{\lambda} \cap \partial B| \neq 0$ separately, since all shearlet
elements $\psi_{\lambda}$ intersect the boundary of the set $B$.
In fact, one needs to first localize the cartoon-like image $f$ by
compactly supported smooth window functions associated with
dyadic squares using a partition of unity. Letting $f_Q$ denote such
a localized version, we then  estimate
$\innerprod{f_Q}{\psi_{\lambda}}$ instead of directly estimating the
shearlet coefficients $\innerprod{f}{\psi_{\lambda}}$.
Moreover, in the case of band-limited shearlets, one needs
to estimate the sparsity of the sequence of the shearlet coefficients rather than
analyzing the decay of individual coefficients.

In the next subsections we present the proof -- first for band-limited, then for compactly
supported shearlets -- in the case $L=1$, i.e., when the discontinuity
curve in the model of cartoon-like images is smooth. Finally, the
extension to $L \neq 1$ will be discussed for both cases
simultaneously.

We will first, however, introduce some notation used in the proofs and
prove a helpful lemma which will be used in both cases: band-limited and
compactly supported shearlets.  For a fixed $j$, we let $\mathcal{Q}_j$
be a collection of dyadic squares defined by
\[
\mathcal{Q}_j = \{Q = [\tfrac{l_1}{2^{j/2}}, \tfrac{l_1+1}{2^{j/2}}] \times [\tfrac{l_2}{2^{j/2}}, \tfrac{l_2+1}{2^{j/2}}] : l_1,l_2 \in \Z\}.
\]
We let $\varLambda$ denote the set of all
indices $(j,k,m)$ in the shearlet system and define
\[\varLambda_j = \{(j,k,m) \in \varLambda:
-\ceilsmall{2^{j/2}} \leq k \leq \ceilsmall{2^{j/2}}, m \in \Z^2\}.\]
For $\eps >0$, we define the set of `relevant' indices on scale $j$ as
\[
\varLambda_j(\eps) = \{ \lambda \in \varLambda_j :
|\innerprod{f}{\psi_{\lambda}}| > \eps\}
\]
and, on all scales, as
\[
\varLambda(\eps) = \{ \lambda \in \varLambda :
|\innerprod{f}{\psi_{\lambda}}| > \eps\}.
\]
\begin{lemma}\label{lem:suff-Lambda-eps-cond}
  Assume Setup~\ref{setup:assumptions-on-generators-BL-2D}
  or~\ref{setup:assumptions-on-generators-CS-2D}. Let $f\in
  \cE^2_L(\R^2)$. Then the following assertions hold:
\begin{enumerate}[(i)]
\item For some constant $C$, we have
\begin{equation}\label{eq:limit_j}
  \card{ \varLambda_j(\eps)} = 0 \quad \text{for} \quad j \ge
  \frac{4}{3}\log_2(\eps^{-1})+C
\end{equation}
\item If \begin{equation} \label{eq:sparse_mainestimate1}
  \card{ \varLambda_j(\eps)} \lesssim \eps^{-2/3},
\end{equation}
  for $j \ge 0$, then
\begin{equation} \label{eq:sparse_mainestimate2}
  \card{ \varLambda(\eps) } \lesssim \eps^{-2/3}\,\log_2(\eps^{-1}),
\end{equation}
which, in turn, implies  (\ref{eq:2d-approx-rate}) and
  (\ref{eq:2d-sparsity-rate}).
\end{enumerate}
\end{lemma}

\begin{proof}
\emph{(i).}  Since $\psi \in L^1(\R^2)$ for both the band-limited and
compactly supported setup, we have that
  \begin{eqnarray}\label{eq:estimate_scale}
    |\innerprod{f}{\psi_{\lambda}}| &=& \Bigl| \int_{\R^2}
    f(x) 2^{\frac{3j}{4}}\psi(S_kA_{2^j}x-m)\D x\Bigr| \nonumber \\
    &\leq& 2^{\frac{3j}{4}}\norm[\infty]{f}\int_{\R^2}|\psi(S_kA_{2^j}x-m)|\D x \nonumber \\
    &=& 2^{-\frac{3j}{4}}\norm[\infty]{f}\norm[1]{\psi}.
  \end{eqnarray}
  As a consequence, there is a scale $j_{\eps}$ such that
  $|\innerprod{f}{\psi_{\lambda}}| < \eps$ for each $j \ge j_{\eps}$.
  It therefore follows from \eqref{eq:estimate_scale} that
  \begin{equation*} 
    \card{ \varLambda(\eps)} = 0 \quad \text{for} \quad j >
    \frac{4}{3}\log_2(\eps^{-1})+C.
  \end{equation*}

  \noindent \emph{(ii).} By assertion~(i) and
  estimate~\eqref{eq:sparse_mainestimate1}, we have that
  \begin{equation*} 
    \card{ \varLambda(\eps) } \leq C\; \eps^{-2/3}\,\log_2(\eps^{-1}).
  \end{equation*}
  From this, the value $\eps$ can be written as a function of the
  total number of coefficients $n=\card{ \varLambda(\eps)}$. We obtain
  \[
  \eps(n) \leq C\; n^{-3/2}(\log_2(n))^{3/2} \quad \text{for
    sufficiently large} \,\, n.
  \]
  This implies that
  \[
  \abs{c^\ast_{\,n}} \leq C\; n^{-3/2}(\log_2(n))^{3/2}
  \]
  and
  \[
  \sum_{n > N} \abs{c^\ast_{\,n}}^2 \leq C\; N^{-2}(\log_2(N))^3
  \qquad \text{for sufficiently large} \,\, N>0,
  \]
  where $c^\ast_{\,n}$ as usual denotes the $n$th largest shearlet
  coefficient in modulus. 
\end{proof}

\subsubsection{Proof for Band-Limited Shearlets for $L=1$}
\label{sec:proof-band-limited-2d}

Since we assume $L=1$, we have that $f\in \cE^2_L(\R^2)=\cE^2(\R^2)$.
As mentioned in the previous section, we will now measure the sparsity of
the shearlet coefficients
$\setprop{\innerprod{f}{\mathring{\psi}_\lambda}}{\lambda \in
  \varLambda}$. For this, we will use the weak $\ell^p$ quasi norm
$\norm[w\ell^p]{\cdot}$ defined as follows. For a sequence
$s=(s_i)_{i\in I}$, we let, as usual, $s^\ast_{\,n}$ be the $n$th
largest coefficient in $s$ in modulus. We then define:
\[
\norm[w\ell^p]{s} = \sup_{n>0} \, n^{\frac{1}{p}} \abs{s^\ast_{\,n}} .
\]
One can show \cite{stein_weiss} that this definition is equivalent to
\[
\norm[w\ell^p]{s} = \Bigl(\sup\setpropbig{ \card{\{i:|s_{i}|>\eps\}}\eps^p}{\eps > 0}\Bigr)^{\frac{1}{p}}.
\]

We will only consider the case $\mathring{\psi}=\psi$ since the case
$\mathring{\psi}=\tilde\psi$ can be handled similarly. To analyze the
decay properties of the shearlet coefficients
$(\innerprod{f}{\psi_{\lambda}})_\lambda$ at a given scale parameter $j \ge
0$, we smoothly localize the function $f$ near dyadic squares. Fix the
scale parameter $j \ge 0$. For a non-negative $C^{\infty}$ function
$w$ with support in $[0,1]^2$, we then define a smooth partition of
unity
\[
\sum_{Q \in \mathcal{Q}_j} w_{Q}(x) = 1, \qquad x \in \R^2,
\]
where, for each dyadic square $Q \in \mathcal{Q}_j$, $w_{Q}(x) = w(2^{j/2}x_1-l_1,2^{j/2}x_2-l_2)$.
We will then examine the shearlet coefficients of the localized function $f_{Q} := fw_{Q}$.
With this smooth localization of the function $f$, we can now consider the two separate cases,
$|\supp w_Q  \cap \partial B| = 0$ and $|\supp w_Q \cap \partial B| \neq 0$.
Let
\[
\mathcal{Q}_j = \mathcal{Q}^0_j \cup \mathcal{Q}_j^1,
\]
where the union is disjoint and
$\mathcal{Q}_j^0$ is the collection of those dyadic squares $Q \in \mathcal{Q}_j$ such that
the edge curve $\partial B$ intersects the support of $w_Q$. Since each $Q$ has side length $2^{-j/2}$ and
the edge curve $\partial B$ has finite length, it follows that
\begin{equation} \label{eq:Q0}
\cardsmall{ \mathcal{Q}_j^0} \lesssim 2^{j/2}.
\end{equation}
Similarly, since $f$ is compactly supported in
$[0,1]^2$, we see that
\begin{equation} \label{eq:Q1}
\cardsmall{ \mathcal{Q}^1_j} \lesssim 2^{j}.
\end{equation}
The following theorems analyzes the sparsity of the shearlets coefficients for each dyadic square $Q \in \mathcal{Q}_j$.
\begin{theorem}[\cite{GL07}]\label{thm:nonsmooth}
Let $f \in \cE^2(\R^2)$.  For $Q \in \mathcal{Q}_j^0$, with $j \ge 0$ fixed, the sequence of shearlet coefficients $\{d_\lambda:=\innerprod{f_Q}{\psi_{\lambda}}:\lambda \in \varLambda_j\}$ obeys
$$
\norm[w\ell^{2/3}]{(d_\lambda)_{\lambda\in\varLambda_j}} \lesssim  2^{-\frac{3j}{4}}.
$$
\end{theorem}
\begin{theorem}[\cite{GL07}]\label{thm:smooth}
Let $f \in \cE^2(\R^2)$. For $Q \in \mathcal{Q}_j^1$, with $j \ge 0$ fixed, the sequence of shearlet coefficients $\{d_\lambda:=\innerprod{f_Q}{\psi_{\lambda}}:\lambda \in \varLambda_j\}$ obeys
$$
\norm[w\ell^{2/3}]{(d_\lambda)_{\lambda\in\varLambda_j}} \lesssim  2^{-\frac{3j}{2}}.
$$
\end{theorem}

As a consequence of these two theorems, we have the following result.
\begin{theorem}[\cite{GL07}]\label{thm:both}
Suppose $f \in \cE^2(\R^2)$. Then, for $j \ge 0$, the sequence of the shearlet coefficients
$\{c_\lambda:=\innerprod{f}{\psi_{\lambda}}:\lambda \in \varLambda_j\}$ obeys
$$
\norm[w\ell^{2/3}]{(c_\lambda)_{\lambda\in\varLambda_j}} \lesssim 1.
$$
\end{theorem}
\begin{proof}
Using Thm.~\ref{thm:nonsmooth} and \ref{thm:smooth}, by the $p$-triangle inequality for weak
$\ell^p$ spaces, $p \leq 1$, we have
\begin{eqnarray*}
\norm[w\ell^{2/3}]{\innerprod{f}{\psi_{\lambda}}}^{2/3} &\leq&
\sum_{Q \in \mathcal{Q}_j}
\norm[w\ell^{2/3}]{\innerprod{f_Q}{\psi_{\lambda}}}^{2/3} \\
&=&
\sum_{Q \in \mathcal{Q}^0_j}\norm[w\ell^{2/3}]{\innerprod{f_Q}{\psi_{\lambda}}}^{2/3} +
\sum_{Q \in \mathcal{Q}^1_j}\norm[w\ell^{2/3}]{\innerprod{f_Q}{\psi_{\lambda}}}^{2/3}
\\
&\leq& C\:\: \card{ \mathcal{Q}^0_j} \: 2^{-j/2}+C\:\: \card{ \mathcal{Q}^1_j} \: 2^{-j}.
\end{eqnarray*}
Equations \eqref{eq:Q0} and \eqref{eq:Q1} complete the proof.
\end{proof}

We can now prove Thm.~\ref{thm:opt-sparse-2D} for the band-limited setup.

\begin{proof}[Thm.~\ref{thm:opt-sparse-2D} for
  Setup~\ref{setup:assumptions-on-generators-BL-2D}]
From Thm.~\ref{thm:both}, we have that
\begin{equation*} 
  \card{ \varLambda_j(\eps)} \leq C\eps^{-2/3},
\end{equation*}
for some constant $C>0$, which, by Lemma~\ref{lem:suff-Lambda-eps-cond}, completes the proof.
\end{proof}

\subsubsection{Proof for Compactly Supported Shearlets for $L=1$}
\label{sec:proof-comp-supp-2d}

To derive the sought estimates (\ref{eq:2d-approx-rate})
and~(\ref{eq:2d-sparsity-rate}) for dimension $d=2$, we will study two
separate cases: Those shearlet elements $\psi_{\lambda}$ which do not
interact with the discontinuity curve, and those elements which do.
\begin{description}
\item[{\em Case 1}.] \hspace*{-1em} The compact support of the shearlet
  $\psi_{\lambda}$ does not intersect the boundary of the set $B$,
  \ie $\abs{\supp \psi_{\lambda} \cap \partial B} = 0$.
\item[{\em Case 2}.] \hspace*{-1em} The compact support of the shearlet
  $\psi_{\lambda}$ does intersect the boundary of the set $B$, \ie
  $\abs{\supp \psi_{\lambda} \cap \partial B} \neq 0 $.
\end{description}

For \emph{Case 1} we will not be concerned with decay estimates of
single coefficients $\innerprod{f}{\psi_\lambda}$, but with the decay
of sums of coefficients over several scales and all shears and
translations. The frame property of the shearlet system, the
$C^2$-smoothness of $f$, and a crude counting argument of the cardinal
of the essential indices $\lambda$ will be enough to provide the
needed approximation rate. The proof of this is similar to estimates
of the decay of wavelet coefficients for $C^2$ smooth functions. In
fact, shearlet and wavelet frames gives the same approximation decay
rates in this case. Due to space limitation of this exposition,
we will not go into the details of this estimate, but rather focus on
the main part of the proof, \emph{Case 2}.

For \emph{Case 2} we need to estimate each coefficient
$\innerprod{f}{\psi_\lambda}$ individually and, in particular, how
$\abs{\innerprod{f}{\psi_\lambda}}$ decays with scale $j$ and shearing
$k$. Without loss of generality we can assume that $f = f_0 + \chi_B
f_1$ with $f_0=0$.  We let then $M$ denote the area of integration in
  $\innerprod{f}{\psi_\lambda}$, that is,
\[M = \supp \psi_\lambda \cap B.
\]
Further, let $\cL$ be an affine hyperplane (in other and simpler words,
a line in $\R^2$) that intersects $M$ and thereby divides $M$ into two
sets $M_t$ and $M_l$, see the sketch in
Fig.~\ref{fig:trunc-and-lin-part-w-M-H}. We thereby have that
\begin{equation} \label{eq:splitting}
\innerprod{f}{\psi_\lambda} = \innerprod{\chi_{M} f}{\psi_\lambda}
= \innerprod{\chi_{M_t} f}{\psi_\lambda} + \innerprod{\chi_{M_l} f}{\psi_\lambda} .
\end{equation}
The hyperplane will be chosen in such way that the area of $ M_t$ is sufficiently small. In particular, $\are{ M_t}$ should be
small enough so that the following estimate
\begin{equation}
 \abs{\innerprod{\chi_{M_t}f}{\psi_\lambda}} \le \norm[L^\infty]{f}
\norm[L^\infty]{\psi_\lambda} \are{ M_t} \le \mu \, 2^{3j/4}
\are{M_t}\label{eq:truncated-estimate-2d-CS}
\end{equation}
do not violate~(\ref{eq:2d-sparsity-rate}). If the hyperplane $\cL$ is
positioned as indicated in Fig.~\ref{fig:trunc-and-lin-part-w-M-H}, it
can indeed be shown by crudely estimating $\are{M_t}$ that
(\ref{eq:truncated-estimate-2d-CS}) does not violate
estimate~(\ref{eq:2d-sparsity-rate}). We call estimates of this form,
where we have restricted the integration to a small part $M_t$ of $M$,
\emph{truncated} estimates. Hence, in the following we assume that
  \eqref{eq:splitting} reduces to $\innerprod{f}{\psi_\lambda} =
  \innerprod{\chi_{M_l} f}{\psi_\lambda}$.

\begin{figure}[ht]
\centering
\includegraphics{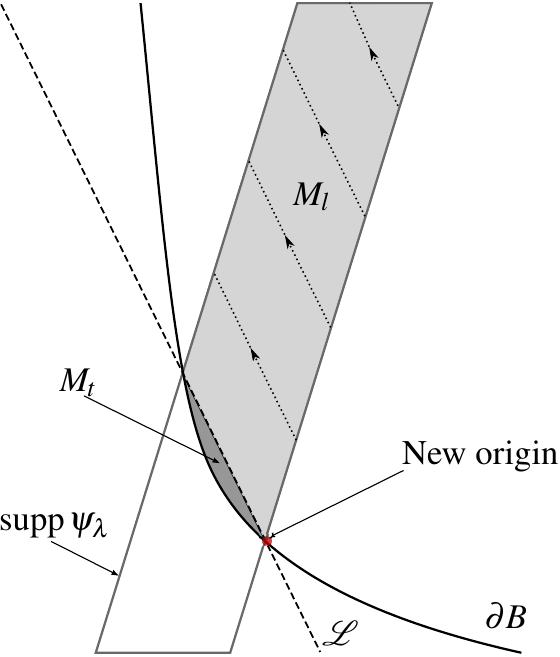}
\caption{Sketch of $\supp \psi_\lambda$, $M_l$, $M_t$, and $\cL$. The
  lines of integrations are shown.}
\label{fig:trunc-and-lin-part-w-M-H}
\end{figure}

For the term $\innerprod{\chi_{M_l} f}{\psi_\lambda}$ we will
have to integrate over a possibly much large part $M_l$ of $M$. To
handle this, we will use that $\psi_\lambda$ only interacts with the
discontinuity of $\chi_{M_l} f$ along a line inside $M$. This part of
the estimate is called the \emph{linearized} estimate, since the
discontinuity curve in $\innerprod{\chi_{M_l} f}{\psi_\lambda}$ has
been reduced to a line. In $\innerprod{\chi_{M_l} f}{\psi_\lambda}$ we
are, of course, integrating over two variables, and we will as the inner
integration always choose to integrate along lines parallel to the
`singularity' line $\cL$, see Fig.~\ref{fig:trunc-and-lin-part-w-M-H}.
The important point here is that along these lines, the
function $f$ is $C^2$-smooth without discontinuities on the entire
interval of integration. This is exactly the reason for removing the
$M_t$-part from $M$. Using the Fourier slice theorem we will then
turn the line integrations along $\cL$ in the spatial domain into line
integrations in the frequency domain. The argumentation is as follows:
Consider $g:\R^2 \to \C$ compactly supported and continuous, and let
$p: \R\to \C$ be a projection of $g$ onto, say, the $x_2$ axis, \ie
$p(x_1)=\int_\R g(x_1,x_2)d x_2$. This immediately implies that $\hat
p(\xi_1) = \hat g(\xi_1,0)$ which is a simplified version of the
Fourier slice theorem. By an inverse Fourier transform, we then have
\begin{equation}
\int_\R g(x_1,x_2)d x_2 = p(x_1) = \int_{\R} \hat
g(\xi_1,0) \expo{2\pi i x_1 \xi_1}
\mathrm{d}\xi_1, \label{eq:fourier-slice-thm}
\end{equation}
and hence
\begin{equation}
 \int_\R  \abs{g(x_1,x_2)}d x_2 = \int_{\R} \abs{\hat
g(\xi_1,0)} \mathrm{d}\xi_1. \label{eq:fourier-slice-thm-2}
\end{equation}
The left-hand side of (\ref{eq:fourier-slice-thm-2}) corresponds
  to line integrations of $g$ along vertical lines
  $x_1=\mathrm{constant}$. By applying shearing to the coordinates $x
  \in \R^2$, we can transform $\cL$ into a line of the form
  $\setprop{x \in \R^2}{x_1=\mathrm{constant}}$, whereby we can
  apply~(\ref{eq:fourier-slice-thm-2}) directly.

  We will make this idea more concrete in the proof of the following
  key estimate for linearized terms of the form
  $\innerprods{\chi_{M_l} f}{\psi_\lambda}$. Since we assume the
  truncated estimate as negligible, this will in fact allow us to
  estimate
  $\innerprod{f}{\psi_\lambda}$. 

\begin{theorem}\label{thm:decay-hyperplane}
  Let $\psi \in L^2(\R^2)$ be compactly supported, and assume that
  $\psi$ satisfies the conditions in
  Setup~\ref{setup:assumptions-on-generators-CS-2D}. Further, let
  $\lambda$ be such that $\supp \psi_\lambda \cap \pa B \neq
  \emptyset$. Suppose that $f \in \cE(\R^2)$
  and that $\pa B$ is linear on the support of $\psi_\lambda$ in the
  sense
\[ \supp \psi_\lambda \cap \pa B \subset \cL \]
for some affine hyperplane $ \cL$ of $\R^2$. Then,
\begin{enumerate}[(i)]
\item if $\cL$ has normal vector  $(-1,s)$ with $|s| \le 3$,
\begin{equation*}
\abs{\innerprod{f}{\psi_{\lambda}}} \lesssim  \frac{2^{-3j/4}}{\abs{k+2^{j/2}s}^{3}},
\end{equation*}
\item if $\cL$ has normal vector  $(-1,s)$ with $|s| \ge 3/2$,
 \begin{equation*}
\abs{\innerprod{f}{\psi_{\lambda}}} \lesssim  2^{-9j/4},
 \end{equation*}
\item if $\cL$ has normal vector  $(0,s)$ with $s\in \R$,
then 
\begin{equation*}
\abs{\innerprod{f}{\psi_{\lambda}}} \lesssim  2^{-11j/4}.
 \end{equation*}
\end{enumerate}
\end{theorem}

\begin{proof}
  Fix $\lambda$, and let $f \in \cE(\R^2)$. We can without loss of
  generality assume that $f$ is only nonzero on $B$.

  {\em Cases (i) and (ii)}. We first consider
  the cases~(i) and (ii). In these cases, the hyperplane can be written as
  \[ \cL = \setprop{x \in \R^2}{\innerprod{x-x_0}{(-1,s)}=0}\] for some $x_0 \in \R^2$. We shear the hyperplane by
  $S_{-s}$ for $s\in \R$ and obtain
\begin{align*}
  S_{-s}\cL &= \setprop{x \in \R^2}{\innerprod{S_s
      x-x_0}{(-1,s)}=0}  \\
&= \setprop{x \in \R^2}{\innerprod{
    x-S_{-s}x_0}{(S_s)^T(-1,s)}=0}  \\
&= \setprop{x \in \R^2}{\innerprod{ x-S_{-s}x_0}{(-1,0)}=0} \\
& =\setprop{x=(x_1,x_2) \in \R^2}{x_1=\hat x_1}, \quad \text{where $\hat x=S_{-s}x_0$,}
\end{align*}
which is a line parallel to the $x_2$-axis. Here the power of
shearlets comes into play, since it will allow us to only consider line
singularities parallel to the $x_2$-axis. Of course, this requires
that we also modify the shear parameter of the shearlet, that is, we
will consider the right hand side of
\[{\innerprod{f}{\psi_{j,k,m}}}=
{\innerprods{f(S_s\cdot)}{\psi_{j,\hat{k},m}}}
\]
with the new shear parameter $\hat k = k + 2^{j/2}s$. The integrand in
$\innerprods{f(S_s\cdot)}{\psi_{j,\hat{k},m}}$ has the singularity
plane exactly located on the line $x_1 = \hat x_1$, \ie on $S_{-s}\cL$.

To simplify the expression for the integration bounds, we will fix a new origin on $S_{-s}\cL$, that is, on $x_1=\hat
x_1$; the $x_2$ coordinate of the new origin will be fixed in the next paragraph. Since $f$ is only nonzero of $B$, the
function $f$ will be equal to zero on one side of $S_{-s}\cL$, say, $x_1 < \hat x_1$. It therefore suffices to estimate
\begin{align*} 
    {\innerprods{f_0(S_s\cdot)\chi_{\Omega}}{\psi_{j,\hat{k},m}}}
\end{align*}
for $f_0 \in C^\beta(\R^2)$ and $\Omega = \R_+ \times \R$. Let us assume that $\hat k <0$. The
other case can be handled similarly.

Since $\psi$ is compactly supported, there exists some $c > 0$ such that $\supp{\psi} \subset \itvcc{-c}{c}^2$. By a
rescaling argument, we can assume $c=1$. Let
  \begin{align} \label{eq:cPjk}
    \cP_{j,k} := \setprop{ x \in \R^2 }{ \abssmall{x_1+2^{-j/2}kx_2} \le 2^{-j}   ,\abs{x_2} \leq 2^{-j/2}},
  \end{align}
  With this notation we have $\supp \psi_{j,k,0} \subset \cP_{j,k}$. We say that the shearlet normal direction of the
  shearlet box $\cP_{j,0}$ is $(1,0)$, thus the shearlet normal of a sheared element $\psi_{j,k,m}$ associated with
  $\cP_{j,k}$ is $(1,k/2^{j/2})$. Now, we fix our origin so that, relative to this new
  origin, it holds that
  \begin{align*}
    \supp \psi_{j,\hat k,m}  \subset \cP_{j,\hat k} +
    (2^{-j},0) =: \tilde{\cP}_{j,k}.
  \end{align*}
Then one face of $\tilde{\cP}_{j,\hat k}$ intersects the origin.

Next, observe that the parallelogram  $\tilde{\cP}_{j,k}$  has sides $x_2 = \pm 2^{-j/2}$,
  \begin{align*}
    2^{j}x_1 + 2^{j/2}\hat k x_2 &= 0,
    \text{ and } \\
    2^{j}x_1 + 2^{j/2}\hat k x_2  &= 2
    .
  \end{align*}
  As it is only a matter of scaling, we replace the right hand side of the last equation with $1$ for simplicity. Solving the two
  last equalities for $x_2$ gives the following lines:
  \begin{align*}
    L_1: \quad x_2&=- \frac{2^{j/2}}{\hat k} x_1,
    \text{ and } \\
    L_2: \quad x_2&=- \frac{2^{j/2}}{\hat k} x_1 + \frac{2^{-j/2}}{\hat k},
  \end{align*}
  We shows that
  \begin{align} \label{eq:hp-claim2-int-domain}
    \abs{\innerprod{f_0(S_s\cdot)\chi_{\Omega}}{\psi_{j,\hat{k},m}}}
    \lesssim \abs{ \int_{0}^{K_1}
      \int_{L_2}^{L_1} f_0(S_s x) \psi_{j,\hat k, m}(x) \, \mathrm{d}x_2 \mathrm{d}x_1,
       }
  \end{align}
  where the upper integration bound for $x_1$ is $K_1 = 2^{-j}- 2^{-j}\hat k$; this follows from solving $L_2$ for $x_1$
  and using that $\abs{x_2} \le 2^{-j/2}$. We remark that the inner integration over $x_2$ is along lines parallel to
  the singularity line $\pa \Omega = \{0\} \times \R$; as mentioned, this allows us to better handle the singularity
  and will be used several times throughout this section.

 We consider the one-dimensional Taylor expansion for $f_0(S_s \cdot)$ at each
  point $x = (x_1,x_2) \in L_2$ in the $x_2$-direction:
  \begin{align*}
    f_0(S_s x) &=
    a(x_1)+b(x_1)\left(x_2+\frac{2^{j/2}}{\hat{k}}
      \ x_1 \right) + c(x_1,x_2)
    \left(x_2+\frac{2^{j/2}}{\hat{k}} \ x_1 \right)^2,
  \end{align*}
  where $a(x_1),b(x_1)$ and $c(x_1,x_2)$ are all bounded in absolute value by $C(1+\abs{s})^2$. Using
  this Taylor expansion in (\ref{eq:hp-claim2-int-domain}) yields
  \begin{align} \label{eq:hp-split-into-I_l}
    \abs{\innerprod{f_0(S_s\cdot)\chi_{\Omega}}{\psi_{j,\hat{k},m}}}
    \lesssim (1+\abs{s})^2 \abs{
      \int_{0}^{K_1 } \sum_{l=1}^{3}I_{l}(x_1) \, \mathrm{d}x_1 },
  \end{align}
  where
  \begin{align}
    I_1(x_1) &= \abs{\int_{L_1}^{L_2} \psi_{j,\hat k,m}(x) \mathrm{d}x_2}, \\
    I_2(x_1) &= \abs{\int_{L_1}^{L_2} (x_2+K_2) \, \psi_{j,\hat k,m}(x)
      \mathrm{d}x_2}, \\
    I_3(x_1) &= \abs{\int_{0}^{-2^{-j/2}/\hat k} (x_2)^2\, \psi_{j,\hat k,m}(x_1,x_2-K_2) \mathrm{d}x_2},
  \end{align}
  and \[K_2 = \frac{2^{j/2}}{\hat{k}} \ x_1.\]
  We next estimate each integral $I_1$ -- $I_3$ separately.

\smallskip

  {\em Integral $I_1$}. We first estimate $I_1(x_1)$. The Fourier slice theorem, see also (\ref{eq:fourier-slice-thm}), yields directly that
  \[
  I_1(x_1) = \absBig{\int_{\R} \psi_{j,\hat k,m}(x) \mathrm{d}x_2} =
  \absBig{\int_{\R^2} \hat \psi_{j,\hat k,m}(\xi_1,0)\, \expo{2\pi
      i x_1 \xi_1} \mathrm{d}\xi_1} .
  \]
  By the assumptions from Setup~\ref{setup:assumptions-on-generators-CS-2D} we
  have, for all $\xi=(\xi_1,\xi_2,\xi_3)\in \R^2$,
  \begin{align*}
    \absbig{\hat \psi_{j,\hat k,m}(\xi)} \lesssim 2^{-3j/4}
    \absbig{h(2^{-j}\xi_1)}
    \left(1+\absBig{\frac{2^{-j/2}\xi_2}{2^{-j}\xi_1} + \hat
        k}\right)^{-\gamma}
   \end{align*}
  for some $h \in L^1(\R)$. Hence, we can continue our estimate of
  $I_1$ by
  \begin{align*}
    I_1(x_1) \lesssim \int_{\R} 2^{-3j/4}
    \absbig{h(2^{-j}\xi_1)} (1 + \abs{\hat k})^{-\gamma}
     \mathrm{d}\xi_1,
  \end{align*}
  and further, by a change of variables,
  \begin{align} \label{eq:I1}
    I_1(x_1) \lesssim \int_{\R} 2^{j/4} \abs{h(\xi_1)} (1
    + \abssmall{\hat k})^{-\gamma}  \mathrm{d}\xi_1
    \lesssim 2^{j/4} (1 + \abssmall{\hat k})^{-\gamma},
  \end{align}
since $h\in L^1(\R)$.

{\em Integral $I_2$}. We start estimating $I_2(x_1)$ by
    \begin{align*}
  I_2(x_1) \le \abs{\int_{\R} x_2\, \psi_{j,\hat k,m}(x) \mathrm{d}x_2} +
\abs{K_2} \abs{\int_{\R} \psi_{j,\hat k,m}(x) \mathrm{d}x_2} =: S_1 + S_2.
\end{align*}
Applying the Fourier slice theorem again and then utilizing the
decay assumptions on $\hat \psi$  yields
\begin{align*}
  S_1 &= \abs{\int_{\R} x_2 \psi_{j,\hat k,m}(x) \mathrm{d}x_2}
\le   \abs{\int_{\R} \left(\frac{\pa}{\pa \xi_2} \hat \psi_{j,\hat
      k,m}\right)(\xi_1,0)\, \expo{2\pi
      i x_1\xi_1} \mathrm{d}\xi_1} \\
& \lesssim  \int_{\R} 2^{-j/2} 2^{-3j/4} \abs{h(2^{-j}\xi_1)} (1
    + \abssmall{\hat k})^{-\gamma}  \mathrm{d}\xi_1
    \lesssim 2^{-j/4} (1 + \abssmall{\hat k})^{-\gamma}.
\end{align*}
Since $\abs{x_1} \le -\hat{k}_1/2^j$, we have $K_2 \le 2^{-j/2}$. The following estimate of $S_2$ then follows directly
from the estimate of $I_1$:
\begin{align*}
  S_2 \lesssim \abs{K_2} 2^{j/4}\, (1 + \abssmall{\hat
    k})^{-\gamma}
\lesssim 2^{-j/4}  \, (1 + \abssmall{\hat k})^{-\gamma}.
\end{align*}
From the two last estimate, we conclude that $ I_2(x_1) \lesssim 2^{-j/4}\,  (1 + \abssmall{\hat k})^{-\gamma}$.

{\em Integral $I_3$}. Finally, we estimate $I_3(x_1)$ by
\begin{align}\nonumber
  I_3(x_1) &\le \abs{\int_{0}^{2^{-j/2}/\hat k} (x_2)^2\,
    \normsmall[L^\infty]{\psi_{j,\hat k,m}} \, \mathrm{d}x_2} \\ \label{eq:I3}
&\lesssim 2^{3j/4} \abs{\int_{0}^{-2^{-j/2}/\hat k} (x_2)^2\,
     \mathrm{d}x_2} \lesssim 2^{-3j/4} \, \abssmall{\hat k}^{-3}.
\end{align}

We see that $I_2$ decays faster than $I_1$, hence we can leave $I_2$ out of our analysis.
Applying \eqref{eq:I1} and \eqref{eq:I3} to (\ref{eq:hp-split-into-I_l}), we obtain
\begin{equation}
 \abs{\innerprod{f_0(S_s\cdot)\chi_{\Omega}}{\psi_{j,\hat{k},m}}}
  \lesssim  (1+\abs{s})^2 \left(
    \frac{2^{-3j/4}}{(1+\abssmall{\hat k})^{\gamma-1}} + \frac{2^{-7j/4}}{\abssmall{\hat
k}^{2}}\right). \label{eq:hp-claim2-est-with-shear}
\end{equation}

 Suppose that $s \le 3$. Then
 (\ref{eq:hp-claim2-est-with-shear}) reduces to
\begin{align*}
  \abs{\innerprod{f}{\psi_{j,k,m}}}
  &\lesssim \frac{2^{-3j/4}}{(1+\abssmall{\hat k})^{\gamma-1}} + \frac{2^{-7j/4}}{\abssmall{\hat
k}^{2}} \\ 
&\lesssim \frac{2^{-3j/4}}{(1+\abssmall{\hat k})^{3}},
\end{align*}
since $\gamma \ge 4$. This proves (i).

On the other hand, if $s \ge 3/2$, then
\begin{align*} 
  \absbig{\innerprod{f}{\psi_{j,k,m}}}
  \lesssim 2^{-9j/4}.
\end{align*}
To see this, note that
$$
\frac{2^{-\frac{3}{4}j}}{(1+|k+s2^{j/2}|)^3} = \frac{2^{-\frac{9}{4}j}}{(2^{-j/2}+|k/2^{-j/2}+s|)^3} \leq  \frac{2^{-\frac{9}{4}j}}{|k/2^{j/2}+s|^3}
$$
and
$$
|k/2^{j/2}+s| \ge |s|-|k/2^{j/2}| \ge 1/2 - 2^{-j/2} \ge 1/4
$$
for sufficiently large $j \ge 0$, since $|k| \leq \ceil{2^{j/2}} \leq 2^{j/2}+1$, and (ii) is proven.

\medskip

{\em Case (iii)}. Finally, we need to consider the case~(iii), in which the normal vector
of the hyperplane $\cL$ is of the form $(0,s)$ for $s \in \R$. For this, let
$\tilde{\Omega}=\setprop{x \in \R^2}{x_2 \ge 0}$. As in the first part
of the proof, it suffices to consider coefficients of the form
$\innerprod{\chi_{\tilde{\Omega}}f_0}{\psi_{j,k,m}}$, where $\supp
\psi_{j,k,m} \subset \cP_{j,k}-(2^{-j},0) = \tilde\cP_{j,k}$ with
respect to some new origin. As before, the boundary of
$\tilde\cP_{j,k}$ intersects the origin. By the assumptions in
Setup~\ref{setup:assumptions-on-generators-CS-2D}, we have that
\[ \left( \frac{\pa}{\pa \xi_1}\right)^\ell \hat \psi(0,\xi_2)=0 \quad \text{for }
\ell =0,1,\]
which implies that
\[ \int_{\R} x_1^\ell \psi(x) \mathrm{d}x_1 = 0 \quad \text{for all }
x_2 \in \R \text{ and } \ell =0,1.\]
Therefore, we have
\begin{align}\label{eq:hp-shear-preserves-vm}
  \int_{\R} x_1^\ell \psi(S_k x) \mathrm{d}x_1 = 0 \quad \text{for all }
  x_2 \in \R, k\in \R, \text{ and } \ell =0,1,
\end{align}
since a shearing operation $S_k$
preserves vanishing moments along the $x_1$ axis. Now,  we employ Taylor
expansion of $f_0$ in the $x_1$-direction (that is,
again along the singularity line $\partial \tilde{\Omega}$).
By (\ref{eq:hp-shear-preserves-vm}) everything but the last term in
the Taylor expansion disappears, and we
obtain
\begin{align*}
  \abs{\innerprod{\chi_{\tilde{\Omega}}f_0}{\psi_{j,k,m}}} &\lesssim
2^{3j/4} \int_{0}^{2^{-j/2}} \int_{-2^{-j}}^{2^{-j}} (x_1)^2 \,\D x_1
\D x_2  \\
& \lesssim 2^{3j/4}\, 2^{-j/2}\, 2^{-3j} = 2^{-11j/4},
\end{align*}
which proves claim (iii).
\end{proof}

We are now ready show the estimates~\eqref{eq:sparse_mainestimate1}
  and \eqref{eq:sparse_mainestimate2}, which by
  Lem.~\ref{lem:suff-Lambda-eps-cond}(ii) completes the proof of
  Thm.~\ref{thm:opt-sparse-2D}.

For $j \ge 0$, fix $Q \in \mathcal{Q}_j^0$, where $\mathcal{Q}_j^0\subset
\mathcal{Q}_j$ is the collection of dyadic squares that intersects
$\cL$.
 We then have the following
counting estimate:
\begin{equation}
\label{eq:counting}
\card{M_{j,k,Q}}  \lesssim \abssmall{k+2^{j/2}s}+1 
\end{equation}
for each $\abs{k} \le \ceil{2^{j/2}}$, where
\[ 
M_{j,k,Q}:=\setprop{m \in \Z^2 }{ |\supp{\psi_{j,k,m}} \cap \cL
    \cap Q| \neq 0}
 \]

 To see this claim, note that for
a fixed $j$ and $k$ we need to count the number of translates $m \in
\Z^2$ for which the support of $\psi_{j,k,m}$ intersects the discontinuity line $\cL:
x_1=sx_2+b$, $b\in \R$, inside $Q$. Without loss of generality, we can
assume that $Q=\itvcc{0}{2^{-j/2}}^2$, $b=0$, and $\supp \psi_{j,k,0}
\subset C \cdot \cP_{j,k}$, where $\cP_{j,k}$ is defined as
in~(\ref{eq:cPjk}). The shearlet $\psi_{j,k,m}$ will therefore be
concentrated around the line $\mathcal{S}_m: x_1=-\frac{k}{2^{j/2}}x_2+
2^{-j}m_1 + 2^{-j/2}m_2$, see also
Fig.~\ref{fig:directional-van-mom-spatial}. 
We will count the number of $m=(m_1,m_2) \in \Z^2$ for which these two
lines intersect inside $Q$ since this number, up to multiplication with a
constant independent of the scale $j$, will be equal to $\cardsmall{M_{j,k,Q}}$. 

First note that since the size of $Q$ is $2^{-j/2} \times 2^{-j/2}$,
only a finite number of $m_2$ translates can make $S_m \cap \cL \cap Q
\neq \emptyset$ whenever $m_1\in \Z$ is fixed. For a fixed $m_2 \in
\Z$, we then estimate the number of relevant $m_1$ translates.
Equating the $x_1$ coordinates in $\cL$ and $\cS_m$ yields
\[ \left(\frac{k}{2^{j/2}} + s\right)x_2 = 2^{-j}m_1 + 2^{-j/2}m_2.\]
Without loss of generality, we take $m_2=0$ which then leads to 
\[ 2^{-j}\abs{m_1} \le 2^{-j/2}\abs{k+2^{j/2}s} \abs{x_2} \le
2^{-j}\abs{k+2^{j/2}s}, \]
hence $\abs{m_1} \le \abs{k+2^{j/2}s}$. This completes the proof of the claim.

For $\eps>0$, we will consider the shearlet coefficients larger
than $\eps$ in absolute value. Thus, we define:
\[
M_{j,k,Q}(\eps) = \setprop{m \in
  M_{j,k,Q}}{\abs{\innerprod{f}{\psi_{j,k,m}}}>\eps},
\]
where $Q \in \cQ^0_j$. Since the discontinuity line $\cL$ has finite
length in $[0,1]^2$, we have the estimate $\cardsmall{ \mathcal{Q}^0_j} \lesssim 
2^{j/2}$. Assume $\cL$ has normal vector $(-1,s)$ with $|s| \leq 3$.
Then, by Thm.~\ref{thm:decay-hyperplane}(i),
$|\innerprod{f}{\psi_{j,k,m}}| > \eps$ implies that
\begin{equation}\label{eq:tool1}
|k+2^{j/2}s| \leq \eps^{-1/3}2^{-j/4}.
\end{equation}
By Lem.~\ref{lem:suff-Lambda-eps-cond}(i) and the estimates~\eqref{eq:counting} and \eqref{eq:tool1}, we have that
\begin{eqnarray*}
  \card{ \varLambda(\eps)} &\lesssim&
  \sum_{j=0}^{\frac{4}{3}\log_2(\eps^{-1})+C}\sum_{Q \in
    \mathcal{Q}^0_j} \; \sum_{\{\hat k: \abssmall{\hat k}\le \eps^{-1/3}2^{-j/4}\}} \card{M_{j,k,Q}(\eps)} \\
  &\lesssim&  \sum_{j = 0}^{\frac{4}{3}\log_2(\eps^{-1})+C}\sum_{Q \in \mathcal{Q}^0_j} \; \sum_{\{\hat k: \abssmall{\hat k}\le \eps^{-1/3}2^{-j/4}\}} (|\hat k|+1) \\
  &\lesssim&  \sum_{j = 0}^{\frac{4}{3}\log_2(\eps^{-1})+C} \card{
    \mathcal{Q}^0_j} \; (\eps^{-2/3}2^{-j/2})\\
  &\lesssim&  \eps^{-2/3} \sum_{j =
    0}^{\frac{4}{3}\log_2(\eps^{-1})+C} 1  \lesssim \eps^{-2/3}\,
  \log_2(\eps^{-1}),
\end{eqnarray*}
where, as usual, $\hat k = k+s2^{j/2}$.  By
Lem.~\ref{lem:suff-Lambda-eps-cond}(ii), this leads to the sought
estimates. 

On the other hand, if $\cL$ has normal vector $(0,1)$ or $(-1,s)$ with
$|s|
\ge 3$, then
$|\innerprod{f}{\psi_{\lambda}}| > \eps$ implies that
\begin{equation*} 
j \leq \frac{4}{9}\log_2(\eps^{-1}),
\end{equation*}
which follows by assertions~(ii) and~(iii) in
Thm.~\ref{thm:decay-hyperplane}. Hence, we have
\[
\card{\varLambda(\eps)} \lesssim
\sum_{j=0}^{\frac{4}{9}\log_2(\eps^{-1})} \sum_{k}\; \sum_{Q \in \mathcal{Q}^0_j} \card{M_{j,k,Q}(\eps)}.
\]
Note that $\card{M_{j,k,Q}} \lesssim 2^{j/2}$, since $\card{ \{m\in
  \Z^2 : |\supp\psi_{\lambda} \cap Q| \neq 0\}} \lesssim 2^{j/2}$ for
each $Q \in \mathcal{Q}_j$, and that
the number of shear parameters $k$ for each scale parameter $j \ge 0$ is bounded by $C2^{j/2}$.
Therefore,
\begin{align*}
  \card{\varLambda(\eps)} \lesssim
  \sum_{j=0}^{\frac{4}{9}\log_2(\eps^{-1})} 2^{j/2}\, 2^{j/2}\, 2^{j/2} = 
\sum_{j=0}^{\frac{4}{9}\log_2(\eps^{-1})} 2^{3j/2} \lesssim
 2^{\frac{4}{9}\cdot \frac{3}{2} \cdot \log_2(\eps^{-1})}
\lesssim  \eps^{-2/3}.
\end{align*}
This implies our sought estimate~\eqref{eq:sparse_mainestimate1}
which, together with the estimate for $\abs{s}\le 3$,
completes the proof of Thm.~\ref{thm:opt-sparse-2D} for $L=1$ under
Setup~\ref{setup:assumptions-on-generators-CS-2D}. 


\subsubsection{The Case $L\neq 1$}
\label{sec:case-l-neq-1}
We now turn to the extended class of cartoon-lime images
$\cE^2_L(\R^2)$ with $L \neq 1$, i.e., in which the singularity curve
is only required to be piecewise $C^2$. We say that $p\in \R^2$ is a
corner point if $\pa B$ is not $C^2$ smooth in $p$. The main focus
here will be to investigate shearlets that interact with one of the
$L$ corner points. We will argue that Thm.~\ref{thm:opt-sparse-2D}
also holds in this extended setting. The rest of the
proof, that is, for shearlets \emph{not} interacting with corner
points, is of course identical to that presented in
Sect.~\ref{sec:proof-band-limited-2d} and~\ref{sec:proof-comp-supp-2d}.

In the compactly supported case one can simply count the number of
shearlets interacting with a corner point at a given scale. Using
Lem.~\ref{lem:suff-Lambda-eps-cond}(i), one then arrives at the
sought estimate. On the other hand, for the band-limited case one
needs  to measure the sparsity of the shearlet coefficients for $f$
localized to each dyadic square. We present the details in the
remainder of this section.

\paragraph{Band-limited Shearlets} 
In this case, it is sufficient to consider a dyadic square $Q \in
\mathcal{Q}_j^0$ with $j \ge 0$ such that $Q$ contains a singular
point of edge curve. Especially, we may assume that $j$ is sufficiently
large so that the dyadic square $Q \in \mathcal{Q}_j^0$ contains a
single corner point of $\partial B$. The following theorem analyzes
the sparsity of the shearlet coefficients for such a dyadic square $Q
\in \mathcal{Q}^0_j$.
\begin{theorem} \label{thm:piece_nonsmooth}
  Let $f \in \cE^2_L(\R^2)$ and $Q \in \mathcal{Q}_j^0$ with $j \ge 0$
  be a dyadic square containing a singular point of the edge curve. The
  sequence of shearlet coefficients
  $\{d_\lambda:=\innerprod{f_Q}{\psi_{\lambda}}:\lambda \in
  \varLambda_j\}$ obeys
\[
\norm[w\ell^{2/3}]{(d_\lambda)_{\lambda\in\varLambda_j}} \leq  C.
\]
\end{theorem}
The proof of Thm.~\ref{thm:piece_nonsmooth} is based on a proof of an
analog result for curvelets \cite{CD04}. Although the proof in
\cite{CD04} considers only curvelet coefficients, essentially the same
arguments, with modifications to the shearlet setting, can be applied
to show Thm.~\ref{thm:piece_nonsmooth}.

Finally, we note that the number of dyadic squares $Q \in
\mathcal{Q}_j^0$ containing a singular point of $\partial B$ is
bounded by a constant not depending on $j$; one could, \eg take $L$
as this constant. Therefore, applying Thm.~\ref{thm:piece_nonsmooth}
and repeating the arguments in Sect.~\ref{sec:proof-band-limited-2d}
completes the proof of Thm.~\ref{thm:opt-sparse-2D} for $L \neq 1$
for Setup~\ref{setup:assumptions-on-generators-BL-2D}.

\paragraph{Compactly Supported Shearlets}
 In this case, it is sufficient to consider the following two cases.
\begin{description}
\item[{\em Case 1}.] \hspace*{-1em} The shearlet $\psi_{\lambda}$ intersects a corner point, in which two $C^2$ curves $\partial B_0$ and $\partial B_1$, say, meet
(see Fig.~\ref{fig:shear-a}).
\item[{\em Case 2}.] \hspace*{-1em} The shearlet $\psi_{\lambda}$ intersects two edge curves $\partial B_0$ and $\partial B_1$, say, simultaneously, but it does not intersect
a corner point (see Fig.~\ref{fig:shear-b}).
\end{description}
\begin{figure}[ht]
\centering
  \begin{minipage}{0.95\textwidth}
    \includegraphics[height=1.4in]{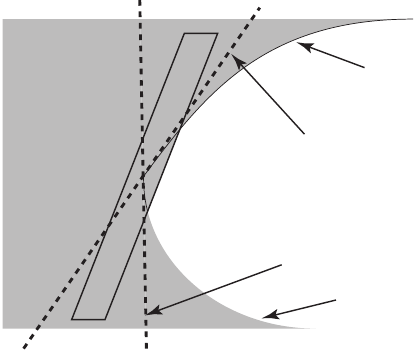}
    \put(-60,44){\footnotesize{$B_1$}}
    \put(-110,70){\footnotesize{$B_0$}}
    \put(-33,28){\footnotesize{$\cL_1$}}
    \put(-30,55){\footnotesize{$\cL_0$}}
    \put(-19,13){\footnotesize{$\partial B_1$}}
    \put(-12,74){\footnotesize{$\partial B_0$}} 
\hspace{7.1em}
    \includegraphics[height=1.4in]{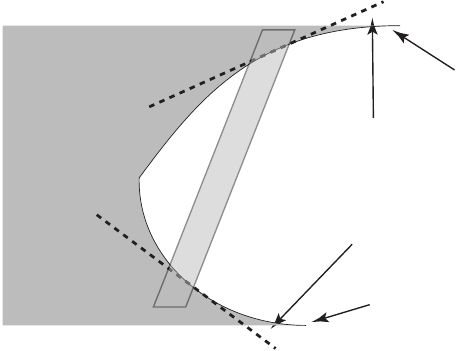}
    \put(-80,44){\footnotesize{$B_1$}}
    \put(-110,70){\footnotesize{$B_0$}}
    \put(-26,32){\footnotesize{$\cL_1$}}
    \put(-27,58){\footnotesize{$\cL_0$}}
    \put(-19,13){\footnotesize{$\partial B_1$}}
    \put(1,72){\footnotesize{$\partial B_0$}}
  \end{minipage}

\begin{minipage}[t]{0.45\textwidth}
    \caption{ A shearlet $\psi_{\lambda}$ intersecting a corner
      point, in which two edge curves $\partial B_0$ and $\partial
      B_1$ meet. $\cL_0$ and $\cL_1$ are tangents to the edge curves
      $\partial B_0$ and $\partial B_1$ in this corner point.}
    \label{fig:shear-a}
  \end{minipage}
\hspace{0.06\textwidth} 
  \begin{minipage}[t]{0.45\textwidth}
    \caption{ A shearlet $\psi_{\lambda}$ intersecting two edge curves
      $\partial B_0$ and $\partial B_1$ which are part of the boundary
      of sets $B_0$ and $B_1$. $\cL_0$ and $\cL_1$ are tangents to the
      edge curves $\partial B_0$ and $\partial B_1$ in points
      contained in the support of $\psi_{\lambda}$.}
    \label{fig:shear-b}
  \end{minipage}
\end{figure}

We aim to show that $\card{\varLambda(\eps)} \lesssim \eps^{-\frac{2}{3}}$ in both cases. By
Lem.~\ref{lem:suff-Lambda-eps-cond}, this will be sufficient.

\subparagraph{Case 1.}
Since there exist only finitely many corner points with total number not depending on scale $j \ge 0$ and the number of
shearlets $\psi_{\lambda}$ intersecting each of corner points is bounded by $C2^{j/2}$, we have
\[
\card{\varLambda(\eps)} \lesssim \sum_{j=0}^{\frac{4}{3}\log_2{(\eps^{-1})}} 2^{j/2} \lesssim  \eps^{-\frac{2}{3}}.
\]

\subparagraph{Case 2.}
As illustrated in Fig.~\ref{fig:shear-b}, we can write the function $f$ as
\[
f_0\chi_{B_0}+f_1\chi_{B_1} = (f_0-f_1)\chi_{B_0}+f_1 \quad \text{in} \,\, Q,
\]
where $f_0,f_1 \in C^2([0,1]^2)$ and $B_0,B_1$ are two disjoint subsets of $[0,1]^2$.
As we indicated before, the rate for optimal sparse approximation is achieved for the smooth
function $f_1$. Thus, it is sufficient to consider $f:= g_0\chi_{B_0}$ with $g_0 = f_0-f_1 \in C^2([0,1]^2)$.
By a \emph{truncated} estimate, we can replace two boundary curves
$\partial B_0$ and $\partial B_1$ by hyperplanes of the form 
\[ \cL_i = \setprop{x \in \R^2}{\innerprod{x-x_0}{(-1,s_i)}=0} \quad \text{for}\,\,i=0,1.\]
In the sequel,
we assume $\max_{i=0,1}|s_i| \leq 3$ and mention that the other cases can be handled similarly.
Next define
$$
M^i_{j,k,Q} = \setprop{m \in \Z^2 }{ |\supp{\psi_{j,k,m}} \cap \cL_i \cap Q| \neq 0} \quad \text{for}\,\, i = 0,1,
$$
for each $Q \in \tilde{\mathcal{Q}^0_j}$, where $\tilde{\mathcal{Q}^0_j}$ denotes the dyadic squares containing the
two distinct boundary curves. By an
estimate similar to \eqref{eq:counting}, we obtain
\begin{equation}\label{eq:piece_counting}
\card{M^0_{j,k,Q} \cap M^1_{j,k,Q}} \lesssim \min_{i=0,1}(|k+2^{j/2}s_i|+1).
\end{equation}
Applying Thm.~\ref{thm:decay-hyperplane}(i) to each of the
hyperplanes $\cL_0$ and $\cL_1$, we also have
\begin{equation}\label{eq:piece_estimate}
|\langle f,\psi_{j,k,m}\rangle| \leq C \cdot \max_{i=0,1} \Bigl\{ \frac{2^{-\frac{3}{4}j}}{|2^{j/2}s_i+k|^3}\Bigr\}.
\end{equation}
Let
$\hat k_i = k+2^{j/2}s_i$ for $i=0,1$.
Without loss of generality, we may assume that $\hat k_0 \leq \hat k_1$. Then, \eqref{eq:piece_counting} and \eqref{eq:piece_estimate} imply that
\begin{equation}\label{eq:piece_counting1}
\card{M^0_{j,Q} \cap M^1_{j,Q}} \lesssim |\hat k_0|+1
\end{equation}
and
\begin{equation}\label{eq:piece_estimate1}
|\langle f,\psi_{j,k,m}\rangle| \lesssim \frac{2^{-\frac{3}{4}j}}{|\hat k_0|^3}.
\end{equation}
Using \eqref{eq:piece_counting1} and \eqref{eq:piece_estimate1}, we now estimate $\card{\varLambda(\eps)}$ as follows:
\begin{eqnarray*}
\card{ \varLambda(\eps)} 
&\lesssim&  \sum_{j = 0}^{\frac{4}{3}\log_2(\eps^{-1})+C}
\sum_{Q \in   \tilde{\mathcal{Q}^0_j}}\, \sum_{\hat k_0}   (1+|\hat k_0|) \\
&\lesssim&  \sum_{j = 0}^{\frac{4}{3}\log_2(\eps^{-1})+C} \card{
  \tilde{\mathcal{Q}^0_j}} \; (\eps^{-2/3}2^{-j/2}) \lesssim \epsilon^{-2/3}.
\end{eqnarray*}
Note that $\cardsmall{ \tilde{\mathcal{Q}^0_j}} \leq C$ since the number of $Q \in \mathcal{Q}_j$ containing two distinct boundary
curves $\partial B_0$ and $\partial B_1$ is bounded by a constant independent of $j$. The result is proved.


\subsection{Optimal Sparse Approximations in 3D}
\label{sec:extensions-3d}

When passing from 2D to 3D, the complexity
of anisotropic structures changes significantly. In particular, as
opposed to the two dimensional setting, geometric structures of
discontinuities for piecewise smooth 3D functions consist of two
morphologically different types of structure, namely surfaces and
curves. Moreover, as we saw in Sect.~\ref{sec:optim-sparse-appr-2d},
the analysis of sparse approximations in 2D heavily depends  on
reducing the analysis to affine subspaces of $\R^2$. Clearly, these
subspaces always have dimension one in 2D. In dimension three,
however, we have subspaces of dimension one and two, and therefore the analysis
needs to performed on subspaces of the `correct' dimension.

This issue manifests itself when performing the analysis for band-limited shearlets,
since one needs to replace the Radon transform used in
2D with a so-called X-ray transform. For compactly supported shearlets,
one needs to perform the analysis on carefully chosen hyperplanes of
dimension two. This will allow for using estimates from the two dimensional
setting in a slice by slice manner.


As in the two dimensional setting, analyzing the decay of individual
shearlet coefficients $\innerprod{f}{\psi_{\lambda}}$ can be used to
show optimal sparsity for compactly supported shearlets while the
sparsity of the sequence of shearlet coefficients with respect to the
weak $\ell^p$ quasi norm should be analyzed for band-limited
shearlets.

\subsubsection{A Heuristic Analysis}
\label{sec:heuristic-analysis-3d}

As in the heuristic analysis for the 2D situation debated in Sect.~\ref{sec:heuristic-analysis}, we can again split
the proof into similar three cases as shown in Fig.~\ref{fig:three-cases}.
\begin{figure}[ht]
\centering
\subfloat[Sketch of shearlets whose support does not intersect
   the surface $\partial B$.]{%
\parbox[t]{0.29\textwidth}{\centering\includegraphics[height=3cm]{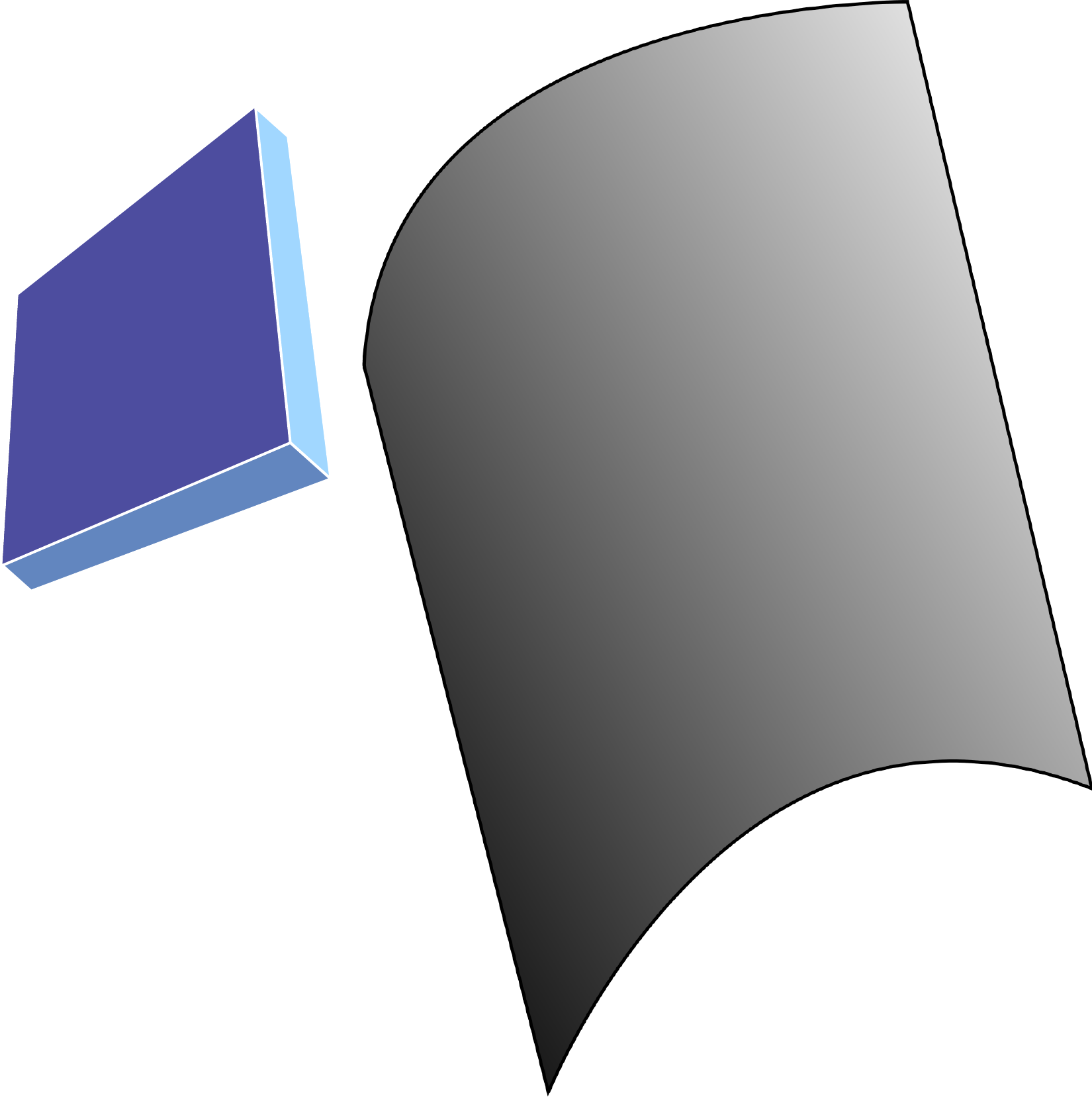}}
\label{fig:case1}
}\quad
\subfloat[Sketch of shearlets whose support overlaps with
    $\partial B$ and is nearly tangent.]{%
\parbox[t]{0.29\textwidth}{\centering\includegraphics[height=3cm]{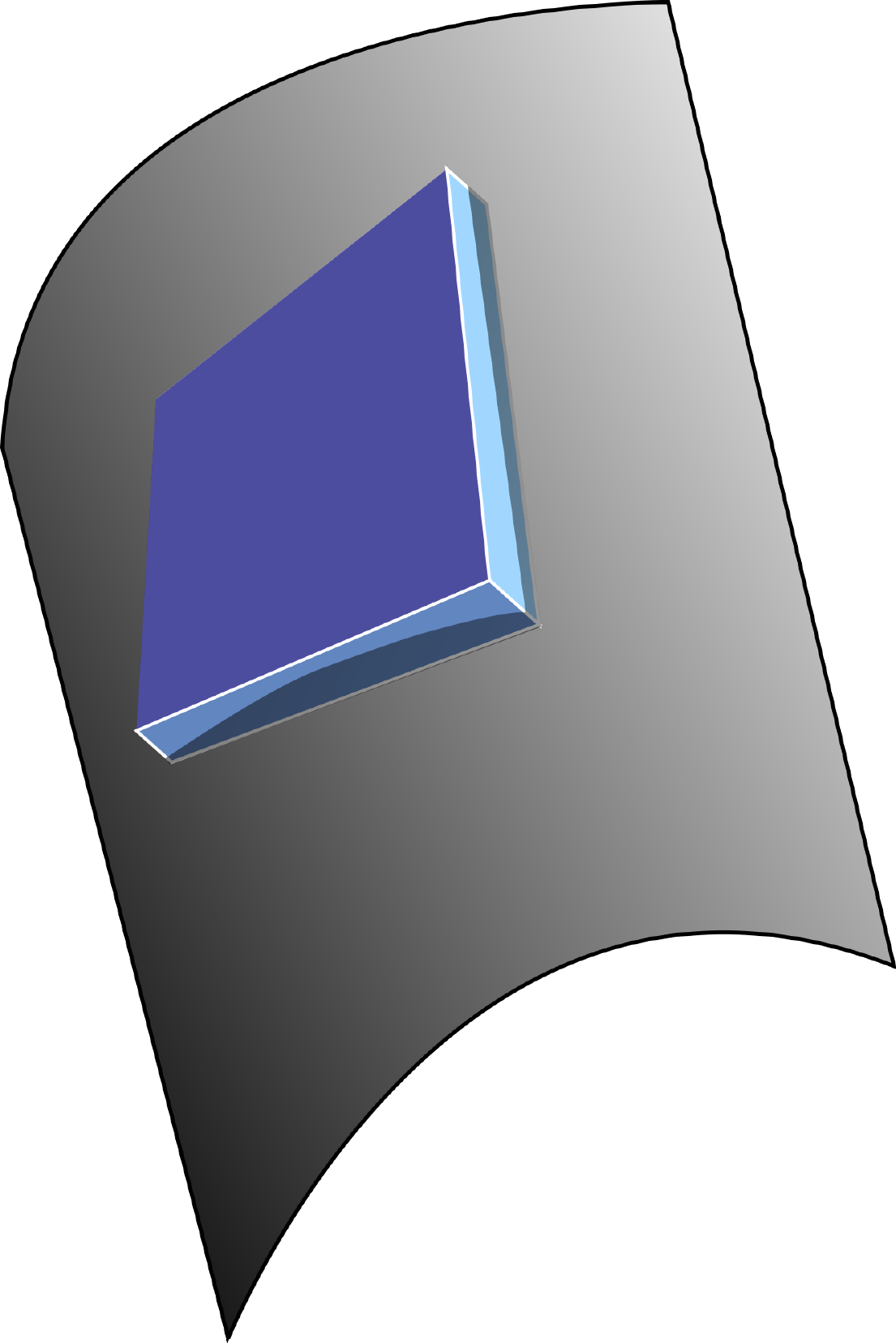}}
\label{fig:case2}
}\quad
\subfloat[Sketch of shearlets whose support overlaps with
    $\partial B$ in a non-tangentially way.]{%
\parbox[t]{0.29\textwidth}{\centering\includegraphics[height=3cm]{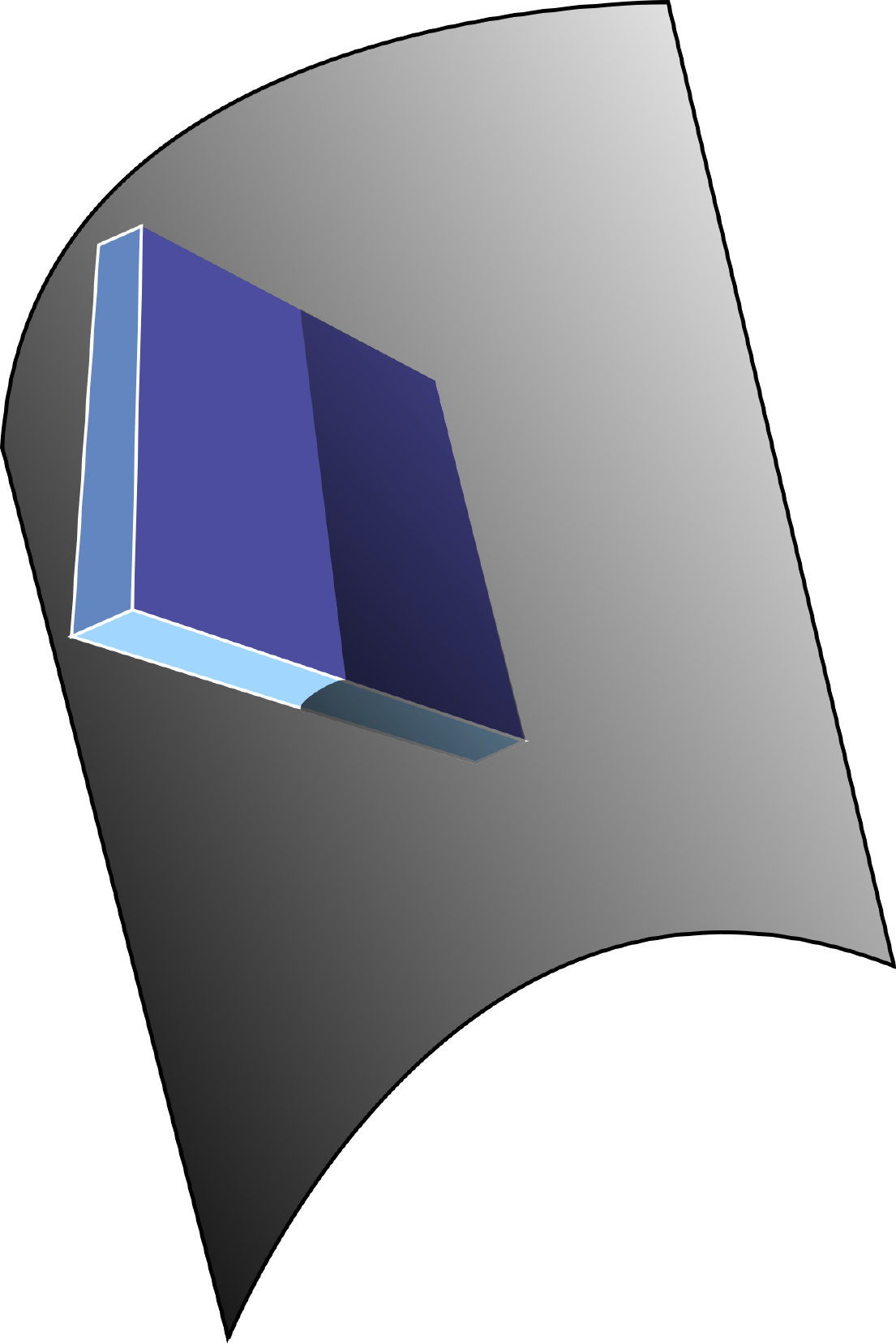}}
\label{fig:case3}
}
\caption{The three types of shearlets $\psi_{j,k,m}$ and boundary
  $\partial B$ interactions considered in the heuristic 3D analysis.
  Note that only a section of $\partial B$ is shown.}
\label{fig:three-cases}
\end{figure}

Only case~(b) differs significantly from the 2D setting, so we
restrict out attention to that case.

For case~(b) there are at most $O(2^{j})$ coefficients at
scale $j>0$, since the plate-like elements are of size $2^{-j/2}$
times $2^{-j/2}$ (and `thickness' $2^{-j}$). By H\"older's inequality,
we see that
\[ \abs{\innerprod{f}{\psi_{j,k,m}}} \le \norm[L^\infty]{f}
  \norm[L^1]{\psi_{j,k,m}} \le C_1 \, 2^{-j}
  \norm[L^1]{\psi} \le C_2 \cdot 2^{-j} \]
for some constants $C_1, C_2>0$. Hence, we have $O(2^{j})$
coefficients bounded by $C_2 \cdot 2^{-j}$.

Assuming the coefficients in case (a) and (c) to be negligible, the
$n$th largest shearlet coefficient $c^\ast_{\,n}$ is therefore bounded
by
\[ \abs{c^\ast_{\,n}} \le C\cdot  n^{-1},\]
which in turn implies
\[
    \sum_{n>N} \abs{c^\ast_{\,n}}^2 \leq \sum_{n>N}C \cdot n^{-2} \le
    C \cdot \int_N^\infty x^{-2} \D x \le C \cdot N^{-1}.
 \]
Hence, we meet the optimal rates~(\ref{eq:approx-fastest-decay-upper})
and (\ref{eq:coeff-fastest-decay-upper}) from Dfn.~\ref{def:optimal}.
This, at least heuristically, shows that shearlets provide optimally
sparse approximations of 3D cartoon-like images.

\subsubsection{Main Result}
\label{sec:main-result-3d}

The hypotheses needed for the band-limited case, stated in
Setup~\ref{setup:assumptions-on-generators-BL-3D}, are a straightforward
generalization of Setup~\ref{setup:assumptions-on-generators-BL-2D} in
the two-dimensional setting.

\begin{setup}\label{setup:assumptions-on-generators-BL-3D}
  The generators $\phi, \psi, \tilde{\psi},\breve{\psi} \in L^2(\R^3)$
  are band-limited and $C^\infty$ in the frequency domain.
  Furthermore, the shearlet system $SH(\phi,\psi,\tilde{\psi},\breve{\psi};c)$ forms a
  frame for $L^2(\R^3)$
  (cf. the construction in
  Sect.~\ref{subsec:bandlimited}).
\end{setup}

For the compactly supported generators we will also use hypotheses in the spirit
of Setup~\ref{setup:assumptions-on-generators-CS-2D}, but with slightly stronger
and more sophisticated assumption on vanishing moment property of the
generators \ie $\delta > 8$ and $\gamma \ge 4$.

\begin{setup}\label{setup:assumptions-on-generators-CS-3D}
  The generators $\phi, \psi, \tilde{\psi}, \breve{\psi} \in
  L^2(\R^3)$ are compactly supported, and the shearlet system
  $SH(\phi,\psi,\tilde{\psi},\breve{\psi};c)$ forms a frame for
  $L^2(\R^3)$. Furthermore, the function $\psi$ satisfies, for all $\xi =
  (\xi_1,\xi_2,\xi_3) \in \R^3$,
      \begin{enumerate}[(i)]
      \item $|\hat\psi(\xi)| \le C \cdot \min\{1,|\xi_1|^{\delta}\}
        \,
        \min\{1,|\xi_1|^{-\gamma}\}
        \, \min\{1,|\xi_2|^{-\gamma}\} \,
        \min\{1,|\xi_3|^{-\gamma}\}$, and
      \item $\left|\frac{\partial}{\partial \xi_i}\hat
          \psi(\xi)\right| \le |h(\xi_1)|
        \left(1+\frac{|\xi_2|}{|\xi_1|}\right)^{-\gamma}
        \left(1+\frac{|\xi_3|}{|\xi_1|}\right)^{-\gamma},$
      \end{enumerate}
      for $i=2,3$, where $\delta > 8$, $\gamma \ge 4$, $h \in
      L^1(\R)$, and $C$ a constant, and $\tilde{\psi}$
      and $\breve{\psi}$ satisfy analogous conditions with the obvious
      change of coordinates (cf. the construction in
      Sect.~\ref{subsec:compactsupport}).
  \end{setup}

The main result can now be stated as follows.

\begin{theorem}[\cite{KLL10, GL10_3d}]
\label{thm:opt-sparse-3D}
Assume Setup~\ref{setup:assumptions-on-generators-BL-3D} or \ref{setup:assumptions-on-generators-CS-3D}. Let $L=1$.
For any $\cp > 0$ and $\mu >0$, the
shearlet frame $SH(\phi,\psi,\tilde{\psi},\breve{\psi};c)$
provides optimally sparse approximations of functions $f \in
\cE_L^2(\R^3)$ in the sense of Dfn. \ref{def:optimal}, i.e., 
\begin{align*}
  \norm[L^2]{f-f_N}^2 &\lesssim N^{-1}(\log{N})^2),  &\text{as $N \to \infty$,}
\intertext{and}
   \abs{c^\ast_{\,n}} &\lesssim n^{-1} (\log
   n),  &\text{as $n \to \infty$,} 
\end{align*}
 where $c=\setprop{\innerprod{f}{\mathring{\psi}_\lambda}}{\lambda \in
     \varLambda, \mathring{\psi}=\psi, \mathring{\psi}=
     \tilde\psi, \text{ or } \mathring{\psi}=
     \breve \psi}$ and $c^\ast=(c^\ast_n)_{n\in \N}$ is a decreasing
   (in modulus) rearrangement of $c$.
\end{theorem}


We now give a sketch of proof for this theorem, and refer to \cite{KLL10, GL10_3d} for detailed proofs.

\subsubsection{Sketch of Proof of Theorem~\ref{thm:opt-sparse-3D}}
\label{sec:sketch-proof-theorem-3d}

\paragraph{Band-limited Shearlets} The proof of
Thm.~\ref{thm:opt-sparse-3D} for band-limited shearlets follows the
same steps as discussed in Sect.~\ref{sec:proof-band-limited-2d} for
the 2D case. To indicate the main steps, we will use the same notation as for the 2D proof
with the straightforward extension to 3D.

Similar to Thm.~\ref{thm:nonsmooth} and \ref{thm:smooth}, one can
prove the following results on the sparsity of the shearlets
coefficients for each dyadic square $Q \in \mathcal{Q}_j$.
\begin{theorem}[\cite{GL10_3d}]\label{thm:nonsmooth-3d}
Let $f \in \cE^2(\R^3)$.  $Q \in \mathcal{Q}_j^0$, with $j \ge 0$ fixed, the sequence of shearlet coefficients $\{d_\lambda:=\innerprod{f_Q}{\psi_{\lambda}}:\lambda \in \varLambda_j\}$ obeys
\[
\normsmall[w\ell^{1}]{(d_\lambda)_{\lambda\in\varLambda_j}} \lesssim  2^{-2j}.
\]
\end{theorem}
\begin{theorem}[\cite{GL10_3d}]\label{thm:smooth-3d}
Let $f \in \cE^2(\R^3)$.  For $Q \in \mathcal{Q}_j^1$, with $j \ge 0$ fixed, the sequence of shearlet coefficients $\{d_\lambda:=\innerprod{f_Q}{\psi_{\lambda}}:\lambda \in \varLambda_j\}$ obeys
\[
\normsmall[\ell^{1}]{(d_\lambda)_{\lambda\in\varLambda_j}} \lesssim  2^{-4j}.
\]
\end{theorem}
The proofs of Thm.~\ref{thm:nonsmooth-3d} and \ref{thm:smooth-3d}
follow the same principles as the proofs of the analog
results in 2D, Thm.~\ref{thm:nonsmooth} and \ref{thm:smooth}, with one
important difference: In the proof of Thm.~\ref{thm:nonsmooth} and
\ref{thm:smooth} the Radon transform (cf.
(\ref{eq:fourier-slice-thm})) is used to deduce estimates for the
integral of edge-curve fragments. In 3D one needs to use a different
transform, namely the so-called X-ray transform, which maps a function
on $\R^3$ into the sets of its line integrals. The X-ray transform is then used
to deduce estimates for the integral of the \emph{surface} fragments. We refer to
\cite{GL10_3d} for a detailed exposition.

As a consequence of Thm.~\ref{thm:nonsmooth-3d} and~\ref{thm:smooth-3d}, we have the following result.
\begin{theorem}[\cite{GL10_3d}]\label{thm:both-3d}
Suppose $f \in \cE^2(\R^3)$. Then, for $j \ge 0$, the sequence of the shearlet coefficients
$\{c_\lambda:=\innerprod{f}{\psi_{\lambda}}:\lambda \in \varLambda_j\}$ obeys
$$
\normsmall[w\ell^{1}]{(c_\lambda)_{\lambda\in\varLambda_j}} \lesssim 1.
$$
\end{theorem}
\begin{proof}
  The result follows by the same arguments used in the proof of Thm.~\ref{thm:both}.
\end{proof}
By Thm.~\ref{thm:both-3d}, we can now  prove
Thm.~\ref{thm:opt-sparse-3D} for the band-limited setup and for $f\in
\cE^2_L(\R^3)$ with $L=1$. The proof is
very similar to the proof of Thm.~\ref{thm:opt-sparse-2D} in
Sect.~\ref{sec:proof-band-limited-2d}, wherefore we will not repeat it.


\paragraph{Compactly Supported Shearlets}
In this section we will consider the key estimates for the linearized term for
compactly supported shearlets in 3D. This is an extension of
Thm.~\ref{thm:decay-hyperplane} to the three-dimensional setting.
Hence, we will assume that the discontinuity surface is a plane, and
consider the decay of the shearlet coefficients of shearlets
interacting with such a discontinuity.

\begin{theorem}[\cite{KLL10}]
\label{thm:decay-hyperplane-3d}
  Let $\psi \in L^2(\R^3)$ be compactly supported, and assume that
  $\psi$ satisfies the conditions in
  Setup~\ref{setup:assumptions-on-generators-CS-3D}. Further, let
  $\lambda$ be such that $\supp \psi_\lambda \cap \pa B \neq
  \emptyset$. Suppose that $f \in \cE^2(\R^3)$ and that $\pa B$ is
  linear on the support of $\psi_\lambda$ in the sense that
\[ \supp \psi_\lambda \cap \pa B \subset \cH \]
for some affine hyperplane $\cH$ of $\R^3$. Then,
\begin{enumerate}[(i)]
\item if $\cH$ has normal vector  $(-1,s_1,s_2)$ with $s_1 \le 3$
  and $s_2 \le 3$,
\begin{equation*} 
\abs{\innerprod{f}{\psi_{\lambda}}} \lesssim \min_{i=1,2} \left\{
  \frac{2^{-j}}{\abs{k_i+2^{j/2}s_i}^{3}} \right\},
\end{equation*}
\item if $\cH$ has normal vector  $(-1,s_1,s_2)$ with $s_1 \ge 3/2$
  or $s_2 \ge 3/2$,
 \begin{equation*} 
\abs{\innerprod{f}{\psi_{\lambda}}} \lesssim  2^{-5j/2},
 \end{equation*}
\item if $\cH$ has normal vector  $(0,s_1,s_2)$ with $s_1,s_2 \in \R$,
then
 \begin{equation*} 
\abs{\innerprod{f}{\psi_{\lambda}}} \lesssim  2^{-3j},
 \end{equation*}
\end{enumerate}
\end{theorem}

\begin{proof}
    Fix $\lambda$, and let $f \in \cE^2(\R^3)$. We first consider
  the case~(ii) and assume $s_1 \ge 3/2$. The hyperplane can be written as
  \[ \cH = \setprop{x \in \R^3}{\innerprod{x-x_0}{(-1,s_1,s_2)}=0}\] for
  some $x_0 \in \R^3$. For $\hat x_3 \in \R$, we consider the
  restriction of $\cH$ to the slice $x_3=\hat x_3$. This is clearly a line
  of the form
  \[ \cL = \setprop{x=(x_1,x_2) \in \R^2}{\innerprod{x-x_0'}{(-1,s_1)}=0}\] for
  some $x_0' \in \R^2$, hence we have reduced the singularity to a line singularity,
  which was already considered in Thm.~\ref{thm:decay-hyperplane}.
  We apply now Thm.~\ref{thm:decay-hyperplane} to each on slice, and
  we obtain
  \[ \abs{\innerprod{f}{\psi_{\lambda}}} \lesssim 2^{j/4}\, 2^{-9j/4}\,
  2^{-j/2} = 2^{-5j/2}.
 \]
 The first term $2^{j/4}$ in the estimate above is due to the
 different normalization factor used for shearlets in 2D and 3D, the second term is
 the conclusion from Thm.~\ref{thm:decay-hyperplane}, and the third is
 the length of the support of $\psi_\lambda$ in the direction of
 $x_3$. The case $s_2 \ge 3/2$ can be handled similarly with restrictions to slices
 $x_2=\hat x_2$ for $\hat x_2 \in \R$. This completes the proof of
 case~(ii).

The other two cases, i.e., case (i) and (ii), are proved using the same slice by slice technique
and Thm.~\ref{thm:decay-hyperplane}.
\end{proof}

Neglecting truncated estimates, Thm.~\ref{thm:decay-hyperplane-3d}
  can be used to prove the optimal sparsity result in
  Thm.~\ref{thm:opt-sparse-3D}. The argument is similar to the one in
  Sect.~\ref{sec:proof-comp-supp-2d} and will not be repeated here.
  Let us simply argue that the decay rate
$\abs{\innerprod{f}{\psi_{\lambda}}} \lesssim  2^{-5j/2}$ from
Thm.~\ref{thm:decay-hyperplane-3d}(ii) is what is
needed in the case $s_i \ge 3/2$.
It is easy to see that in 3D an estimate of the form
\[ \card{\varLambda (\eps)} \lesssim \eps^{-1}.\] will guarantee
optimal sparsity. Since we in the estimate
$\abs{\innerprod{f}{\psi_{\lambda}}} \lesssim 2^{-5j/2}$ have no
control of the shearing parameter $k=(k_1,k_2)$, we have to use a
crude counting estimate, where we include all shears at a given scale
$j$, namely $2^{j/2} \cdot 2^{j/2} = 2^{j}$. Since the number of
dyadic boxes $Q$ where $\partial B$ intersects the support of $f$ is
of order $2^{3j/2}$, we arrive at
\[ \card{\varLambda (\eps)} \lesssim
\sum_{j=0}^{\tfrac{2}{5}\log_2(\eps^{-1})} 2^{5j/2} \asymp \eps^{-1}. \]

\subsubsection{Some Extensions}
\label{sec:some-extensions}

Paralleling the two-dimensional setting (see Sect.~\ref{sec:case-l-neq-1}),
we can extend the optimality result in Thm.~\ref{thm:opt-sparse-3D} to
the cartoon-like image class $\cE_L^2(\R^3)$ for $L\in \N$, in which the
discontinuity surface $\partial B$ is allowed to be \emph{piecewise}
$C^2$ smooth.

Moreover, the requirement that the `edge' $\partial B$ is piecewise
$C^2$ might be too restrictive in some applications. Therefore,
in \cite{KLL10}, the cartoon-like image model class was enlarged to
allow less regular images, where $\partial B$ is piecewise $C^{\alpha}$
smooth for $1 < \alpha \le 2$, and not necessarily a $C^2$. This class
$\cE^\beta_{\alpha,L}(\R^3)$ was introduced in
Sect.~\ref{sec:cartoon-images} consisting of \emph{generalized}
cartoon-like images having $C^\beta$ smoothness apart from a piecewise
$C^\alpha$ discontinuity curve. The sparsity results presented above
in Thm.~\ref{thm:opt-sparse-3D} can be extended to this generalized model
class for compactly supported
shearlets with a scaling matrix dependent on $\alpha$. The optimal
approximation error rate, as usual measured in $\norm[L^2]{f - f_N
}^2$, for this generalized model is $N^{-\alpha/2}$; compare this
to $N^{-1}$ for the case $\alpha=2$ considered throughout this
chapter. For brevity we will not go into details of this, but mention
the approximation error rate obtained by shearlet frames
is slightly worse than in the $\alpha=\beta=2$ case, since the error
rate is not only a poly-log factor away from the optimal rate, but a
small polynomial factor; and we refer to~\cite{KLL10} the precise
statement and proof.

\subsubsection{Surprising Observations}
\label{sec:surpr-observ}

Capturing anisotropic phenomenon in 3D is somewhat different from
capturing anisotropic features in 2D as discussed in Sect.~\ref{subsec:3D}. While in 2D we `only' have to
handle curves, in 3D a more complex situation can occur since we find
two geometrically very different anisotropic structures: curves and
surfaces. Curves are clearly one-dimensional anisotropic features and
surfaces two-dimensional features. Since our 3D shearlet elements are
plate-like in spatial domain by construction, one could think that these
3D shearlet systems would {\em only} be able to efficiently capture
two-dimensional anisotropic structures, and {\em not} one-dimensional
structures. Nonetheless, surprisingly, as we have discussed in
Sect.~\ref{sec:some-extensions}, these 3D shearlet systems still
perform optimally when representing and analyzing 3D data
$\cE_L^2(\R^3)$ that contain {\em both} curve and surface singularities (see \eg
Fig.~\ref{fig:cartoon-piecewise}).

{\small \paragraph{Acknowledgements} The first author acknowledges partial support by Deutsche Forsch\-ungs\-gemein\-schaft (DFG) Grant KU 1446/14, and the
first and third author acknowledge support by DFG Grant SPP-1324 KU 1446/13.}


\end{document}